\documentclass{amsart}
\usepackage{amssymb}
\usepackage{mathrsfs}
\usepackage[all]{xy}
\usepackage{color}

\newcommand{\C}{\mathbb{C}}
\newcommand{\F}{\mathbb{F}}
\newcommand{\R}{\mathbb{R}}
\newcommand{\Z}{\mathbb{Z}}

\newcommand{\uF}{\underline{\F}}

\newcommand{\cE}{\mathcal{E}}
\newcommand{\cF}{\mathcal{F}}
\newcommand{\cG}{\mathcal{G}}

\newcommand{\cK}{\mathcal{K}}
\newcommand{\cP}{\mathcal{P}}

\newcommand{\cS}{\mathscr{S}}
\newcommand{\cT}{\mathscr{T}}

\newcommand{\Db}{D^{\mathrm{b}}}
\newcommand{\Dp}{D^{\mathrm{pf}}}

\newcommand{\Kb}{K^{\mathrm{b}}}
\newcommand{\cD}{\mathscr{D}}

\newcommand{\cA}{\mathscr{A}}

\newcommand{\cM}{\mathscr{M}}

\newcommand{\win}{\diamondsuit}
\DeclareMathOperator{\wt}{wt}
\newcommand{\Vect}{\mathrm{Vect}}
\DeclareMathOperator{\Irr}{Irr}
\DeclareMathOperator{\Ind}{Ind}

\newcommand{\MHM}{\mathrm{MHM}}
\newcommand{\MHMs}{\MHM_{\cS}}
\newcommand{\MHMso}{\MHM_{\cS}^\win}
\newcommand{\Perv}{\mathrm{Perv}}
\newcommand{\Pervs}{\Perv_{\cS}}
\newcommand{\Pure}{\mathsf{Pure}}
\newcommand{\MC}{\mathcal{L}}
\newcommand{\IC}{\mathrm{IC}}
\newcommand{\pt}{\mathrm{pt}}
\newcommand{\exs}{\mathrm{rat}}

\newcommand{\degr}{\zeta}
\newcommand{\ds}{\Db_{\cS}}
\newcommand{\dsp}{\Db_{\cS,\Perv}}
\newcommand{\dsm}{\Db_{\cS,\MHM}}

\newcommand{\HH}{\mathcal{H}_{\mathrm{Hodge}}}

\DeclareMathOperator{\gr}{gr}

\newcommand{\const}{{\mathrm{const}}}

\newcommand{\rso}{\rho^\win}
\newcommand{\rsc}{\rho}
\DeclareMathOperator{\Hom}{Hom}
\DeclareMathOperator{\End}{End}
\DeclareMathOperator{\Ext}{Ext}
\DeclareMathOperator{\uHom}{\underline{Hom}}
\DeclareMathOperator{\uEnd}{\underline{End}}

\newcommand{\simto}{\overset{\sim}{\to}}

\newcommand{\gHom}{\mathbf{Hom}}
\newcommand{\hexs}{\widehat\exs}
\newcommand{\hbeta}{\hat\beta}

\DeclareMathOperator{\grHom}{grHom}

\newcommand{\half}{{\mathchoice{\textstyle\frac12}{\frac12}{\frac12}{1/2}}}
\newcommand{\nhalf}{{\mathchoice{\textstyle\frac n2}{\frac n2}{\frac n2}{n/2}}}
\newcommand{\rmod}{\mathrm{mod}\textup{-}}
\newcommand{\rgmod}{\mathrm{gmod}\textup{-}}
\newcommand{\rdgmod}{\mathrm{dgmod}\textup{-}}
\newcommand{\rdggmod}{\mathrm{dggmod}\textup{-}}
\newcommand{\uE}{\underline{E}}
\newcommand{\ucA}{\underline{\boldsymbol{\cA}}}
\newcommand{\ucE}{\underline{\boldsymbol{\cE}}}
\newcommand{\For}{\mathsf{For}}

\newcommand{\U}{\mathsf{U}}

\newcommand{\bE}{\mathbb{E}}

\newcommand{\ualpha}{\underline{\alpha}}
\newcommand{\MHmod}{\mathrm{MHmod}\textup{-}}
\newcommand{\MHmodss}{\mathrm{MHmod}^{\mathrm{ss}}\textup{-}}
\newcommand{\gVect}{\mathrm{gVect}}

\numberwithin{equation}{section}
\newtheorem{thm}{Theorem}[section]
\newtheorem{lem}[thm]{Lemma}
\newtheorem{prop}[thm]{Proposition}
\newtheorem{cor}[thm]{Corollary}
\theoremstyle{definition}
\newtheorem{defn}[thm]{Definition}
\theoremstyle{remark}
\newtheorem{rmk}[thm]{Remark}

\title{Koszul duality and mixed Hodge modules}
\author[P.N. Achar]{Pramod N.~Achar}
\address{266 Lockett Hall\\
  Department of Mathematics\\
  Louisiana State University\\
  Baton Rouge, LA \ 70803\\
  U.S.A.}
\email{pramod@math.lsu.edu}

\author{S. Kitchen}
\address{Mathematisches Facult\"at\\
  Albert-Ludwigs-Universit\"at Freiburg\\
  Eckerstrasse 1,\\
  79104 Freiburg\\
  Germany}
\email{sarah.kitchen@math.uni-freiburg.de}

\thanks{This work was partially supported by NSF Grant No.~DMS-1001594.}

\begin{document}

\begin{abstract}
We prove that on a certain class of smooth complex varieties (those with ``affine even stratifications''), the category of mixed Hodge modules is ``almost'' Koszul: it becomes Koszul after a few unwanted extensions are eliminated.  We also give an equivalence between perverse sheaves on such a variety and modules for a certain graded ring, obtaining a formality result as a corollary.  For flag varieties, these results were proved earlier by Beilinson--Ginzburg--Soergel using a rather different construction.
\end{abstract}

\maketitle

\section{Introduction}
\label{sect:intro}

In their seminal paper on Koszul duality in representation theory~\cite{bgs}, Beilinson, Ginzburg, and Soergel established the Koszulity of two important geometric categories: the category of mixed perverse sheaves on a flag variety over a finite field, and the category of mixed Hodge modules on a flag variety over $\C$.  More precisely, they are each ``almost'' Koszul, in that they contain some unwanted extensions, but once those are removed, what remains is a Koszul category.

A key step in~\cite{bgs} is that of giving a concrete description of the extensions to be removed.  However, the two cases are treated very differently.  For $\ell$-adic perverse sheaves, the description preceding~\cite[Theorem~4.4.4]{bgs} is quite general; it applies to any variety satisfying a couple of axioms (cf.~\cite[Lemma~4.4.1]{bgs}), and the proof of Koszulity uses only general results about \'etale cohomology and homological algebra.  In contrast, for mixed Hodge modules (cf.~\cite[Theorem~4.5.4]{bgs}), the description is a rather opaque condition that makes sense only on the full flag variety of a reductive group.  The resulting category is not canonical (it depends on a choice), and the proof of Koszulity depends on detailed knowledge of the structure of one specific projective object.  As written, this part of~\cite{bgs} does not even apply to partial flag varieties (see Remark~\ref{rmk:ginzburg}, however).

The present paper was motivated by a desire to understand the source of this mismatch.  One way to ``remove extensions'' in an abelian category is to discard some objects, i.e., take a subcategory.  Another way (which we call ``winnowing'') is to \emph{add new morphisms} that split formerly nonsplit extensions.  (See Remark~\ref{rmk:winnow-add}.)  In~\cite{bgs}, the subcategory approach is used for both $\ell$-adic perverse sheaves and mixed Hodge modules, but this turns out to be unnatural for mixed Hodge modules: they are much better suited to winnowing.  One reason is that the failure of Koszulity happens in opposite ways in the two cases; see Remark~\ref{rmk:failure}.

In this paper, we develop an approach to Koszul duality for mixed Hodge modules via winnowing.  We prove two main results, both valid on any variety satisfying a few axioms.  The first, Theorem~\ref{thm:koszul}, states that the winnowing of the category of mixed Hodge modules is Koszul.  Unlike the subcategory constructed in~\cite{bgs}, the winnowing does not depend on any choices.  Nevertheless, on a full flag variety, each subcategory from~\cite{bgs} is canonically equivalent to the winnowing.

In the winnowing construction, the notion of ``underlying perverse sheaf'' is lost, and it becomes a nontrivial task to construct a well-behaved functor (called a ``degrading functor'') from the winnowed category of mixed Hodge modules to the category of perverse sheaves.  This is the content of the second main result, Theorem~\ref{thm:halfdegr}.  As a corollary, we obtain a formality result that generalizes a theorem of Schn\"urer.  It turns out that the degrading functor is not canonical in general.  This is roughly why the~\cite{bgs} construction requires choices; the precise relationship is discussed in Section~\ref{subsect:bgs}.

The paper is organized as follows.  Sections~\ref{sect:homol} and~\ref{sect:prelim} contain generalities on homological algebra and on varieties and sheaves, respectively, and Section~\ref{sect:homext} contains some technical lemmas on various $\Hom$- and $\Ext$-groups.  The Koszulity theorem is proved in Section~\ref{sect:koszul}, and the degrading functor is constructed in Section~\ref{sect:degr}.  Finally, Section~\ref{sect:compat} discusses how to compare the degrading functor to the underlying perverse sheaf of a mixed Hodge module.

\subsection*{Acknowledgments}
The authors are grateful to S.~Riche, W.~Soergel, and an anonymous referee for their helpful insights and suggestions.

\section{Preliminaries on homological algebra}
\label{sect:homol}

\subsection{Triangulated categories and realization functors}

For any triangulated category $\cD$ and any two objects $X, Y \in \cD$, we write
\[
\Hom_\cD^i(X,Y) = \Hom_\cD(X, Y[i]).
\]
If $\cM$ is the heart of a bounded $t$-structure on $\cD$, then, under some mild assumptions~\cite{bbd,b}, one can construct a $t$-exact functor of triangulated categories
\[
\mathrm{real}: \Db\cM \to \cD
\]
whose restriction to $\cM$ is the identity functor.  This functor, called a ``realization functor,'' induces maps
\begin{equation}\label{eqn:exti-real}
\mathrm{real}: \Ext_\cM^i(X,Y) \to \Hom_\cD^i(X,Y)
\end{equation}
for all $i \ge 0$.  (In fact, the map~\eqref{eqn:exti-real} is independent of the choice of realization functor, because $\Ext^i({-},{-})$ is a universal $\delta$-functor.)    According to~\cite[Remarque~3.1.17]{bbd}, these maps enjoy special properties for small values of $i$:
\begin{align}
\mathrm{real} &: \Ext^i_{\cM}(X,Y) \simto \Hom^i_{\cD}(X,Y) &&\text{is an isomorphism for $i = 0$ or $1$,} \label{eqn:ext1-real} \\
\mathrm{real} &: \Ext^2_{\cM}(X,Y) \hookrightarrow \Hom^2_{\cD}(X,Y) && \text{is injective.} \label{eqn:ext2-real}
\end{align}
In this paper, all triangulated categories will be linear over some field $\F$.

\subsection{Mixed and Koszul categories}
\label{subsect:mixed-kos}

Recall that a finite-length abelian category $\cM$ is said to be \emph{mixed}~\cite{bgs} if it is equipped with a function $\wt: \Irr(\cM) \to \Z$ (where $\Irr(\cM)$ denotes the set of isomorphism classes of simple objects) such that
\begin{equation}\label{eqn:ab-mixed}
\Ext^1(L,L') = 0
\qquad
\text{if $L, L'$ are simple objects with $\wt(L') \ge \wt (L)$.}
\end{equation}
The category $\cM$ is \emph{Koszul} if it obeys the stronger condition that
\[
\Ext^i(L,L') = 0
\qquad
\text{if $L, L'$ are simple objects with $\wt(L') \ne \wt(L) - i$.}
\]

Suppose now that $\cM$ is also the heart of a bounded $t$-structure on a triangulated category $\cD$.  Then $\cD$ is said to be \emph{mixed} if
\begin{equation}\label{eqn:tri-mixed}
\Hom^i_\cD(L,L') = 0
\qquad
\text{if $L, L' \in \cM$ are simple and $\wt(L') > \wt (L) - i$.}
\end{equation}
For any two integers $n \le m$, we denote by
\[
\cM_{\le n}, \quad \cM_{\ge n}, \quad \cM_{[n,m]}, \quad \cM_n
\]
the Serre subcategories of $\cM$ generated by simple objects whose weight $w$ satisfies $w \le n$, $w \ge n$, $w \in [n,m]$, or $w =n$, respectively.
 Condition \eqref{eqn:tri-mixed} implies that
\begin{equation}\label{eqn:tri-gen}
\Hom^i_\cD(X,Y) = 0
\qquad
\text{if $X \in \cM_{\le n}$ and $Y \in \cM_{\ge n+1-i}$.}
\end{equation}
In the special case where $\cD = \Db\cM$, both these conditions are implied by~\eqref{eqn:ab-mixed}.

For any object $X \in \cM$ and any $n \in \Z$, there is a functorial short exact sequence
\begin{equation}\label{eqn:mixed-ses}
0 \to X_{\le n} \to X \to X_{\ge n+1} \to 0,
\end{equation}
where every simple composition factor of $X_{\le n}$ (resp.~$X_{\ge n+1}$) has weight${}\le n$ (resp.${}\ge n+1$). We also define
\[
\gr_n X = (X_{\le n})_{\ge n} \cong (X_{\ge n})_{\le n}.
\]
It is easily seen that if $f: X \to Y$ is a morphism such that $\gr_n f = 0$ for all $n$, then $f = 0$.  An object $X$ is \emph{pure of weight $n$} if $X \cong \gr_n X$.  Every pure object is semisimple.


\begin{lem}\label{lem:extn-surj}
Let $\cM$ be a mixed abelian category, and let $A, B \in \cM$ be pure objects of weights $n$ and $n-i$, respectively.  The product map
\[
\bigoplus_{\substack{S \in \Irr(\cM)\\ \wt(S) = n - 1}}
\Ext^1(A, S) \otimes \Ext^{i-1}(S, B) \to \Ext^i(A, B)
\]
is surjective.
\end{lem}
\begin{proof}
Consider an element $\phi \in \Ext^i(A,B)$.  As in any abelian category, there is an object $P$ such that $\phi$ lies in the image of the product map
\[
\Ext^1(A,P) \otimes \Ext^{i-1}(P,B) \to \Ext^i(A,B).
\]
(See, for instance,~\cite[Proposition~3 in \S7.4]{bourbaki}.)  If $P$ happens to be pure of weight $n-1$, then we are finished, since every pure object is semisimple, and $\Ext^\bullet$ commutes with finite direct sums.  

Otherwise, we must show how to replace $P$ by such a pure object.  Let $f: A \to P[1]$ and $g: P \to B[i-1]$ be morphisms in $\Db\cM$ such that $\phi = g[1] \circ f$.  Because $\cM$ is mixed, the composition $A \overset{f}{\to} P[1] \to P_{\ge n}[1]$ vanishes.  Thus, $f$ factors through a map $f': A \to P_{\le n-1}[1]$.  Let $g'$ denote the composition $P_{\le n-1} \to P \overset{g}{\to} B[i-1]$.  Then $g'[1] \circ f' = \phi$.

Now consider the short exact sequence $0 \to P_{\le n-2} \to P_{\le n-1} \to \gr_{n-1}P \to 0$.  By~\eqref{eqn:tri-gen}, the composition $P_{\le n-2} \to P_{\le n-1} \overset{g'}{\to} B[i-1]$ vanishes, so $g'$ factors through a map $g'': \gr_{n-1}P \to B[i-1]$.  Let $f''$ be the composition $A \overset{f'}{\to} P_{\le n-1}[1] \to \gr_{n-1}P[1]$.  We have $\phi = g''[1] \circ f''$.  Thus, $\phi$ lies in the image of the product map $\Ext^1(A, \gr_{n-1}P) \otimes \Ext^{i-1}(\gr_{n-1}P,B) \to \Ext^i(A,B)$,  as desired.
\end{proof}

\subsection{The winnowing construction}
\label{subsect:winnowmix}

Let $\cD$ be a mixed triangulated category with heart $\cM$ as above.  Recall that both are linear over some field $\F$.  We will work with the following full additive subcategory of $\cD$:
\begin{equation}\label{eqn:pure-defn}
\Pure(\cD) = \{ X \in \cD \mid \text{$H^i(X) \in \cM_i$ for all $i \in \Z$} \}.
\end{equation}
Here, ``$H^i({-})$'' denotes cohomology with respect to the given $t$-structure on $\cD$.  It is easy to prove that every object of $\Pure(\cD)$ is semisimple, i.e., a direct sum of shifts of simple objects of $\cM$.  In other words, if we let $\Ind(\Pure(\cD))$ denote the set of isomorphism classes of indecomposable objects of $\cA$, then we have a bijection
\[
\Irr(\cM) \simto \Ind(\Pure(\cD))
\qquad\text{given by}\qquad
L \mapsto L[-\wt L].
\]
In order to apply certain results of~\cite{ar}, we now impose the additional assumptions:
\begin{align}
\End(L) &\cong \F &&\text{for $L \in \cM$ simple, and} \label{eqn:mixedsplit} \\
\dim \Hom(X,Y) &< \infty &&\text{for all $X, Y \in \Pure(\cD)$.} \label{eqn:dimhomfin}
\end{align}
Let us define a ``degree function'' $\deg: \Ind(\Pure(\cD)) \to \Z$ by
\[
\deg(L[-\wt L]) = \wt L \quad\text{for $L \in \Irr(\cM)$.}
\]
This function makes $\Pure(\cD)$ into an \emph{Orlov category} in the sense of~\cite{ar}.  

\begin{defn}
Let $\cM$ and $\cD$ be as above.  The full subcategory
\[
\cM^\win = 
\left\{
X \in \Kb\Pure(\cD) \,\bigg|\,
\begin{array}{c}
\text{$X$ is isomorphic to a bounded complex $C^\bullet$} \\
\text{in which each $C^i \in \Pure(\cD)$ is a sum} \\
\text{of indecomposable objects of degree $-i$}
\end{array}
\right\}
\]
of the homotopy category $\Kb\Pure(\cD)$ is called the \emph{winnowing} of $\cM$.
\end{defn}

\noindent
Note that $\cM^\win$ is not defined intrinsically in terms of $\cM$ alone: it depends on both $\cD$ and on the mixed structure on $\cM$, although this is not reflected in the notation.

According to~\cite[Proposition~5.4 and Corollary~5.5]{ar}, $\cM^\win$ (which was denoted ``$\mathsf{Kos}(\Pure(\cD))$'' in \emph{loc.~cit.}) is the heart of a $t$-structure on $\Kb\Pure(\cD)$.  In fact, $\cM^\win$ is a finite-length category with a natural mixed structure; the simple objects and weight function are given by
\[
\Irr(\cM^\win) = \{ S[\deg S] \mid S \in \Ind(\Pure(\cD)) \}
\qquad\text{and}\qquad
\wt(S[\deg S]) = \deg S.
\]
Moreover, a strong version of~\eqref{eqn:tri-mixed} holds here: by~\cite[Equation~(5.4)]{ar},
\begin{equation}\label{eqn:orlov-mixed}
\Hom^i_{\Kb\Pure(\cD)}(L,L') = 0
\qquad
\begin{array}{@{}c@{}}
\text{if $L, L' \in \cM^\win$ are simple objects} \\
\text{with $\wt(L') \ne \wt (L) - i$.}
\end{array}
\end{equation}

In the special case where $\cD = \Db\cM$, the category $\Pure(\cD)$ coincides with the one denoted $\mathsf{Orl}(\cM)$ in~\cite[Proposition~5.9]{ar}, and if $\cM$ is Koszul to begin with,~\cite[Theorem~5.10]{ar} asserts that $\cM^\win \cong \cM$.  Of course, in the present paper, we will apply this construction in a case where $\cM$ is not Koszul.

\begin{prop}\label{prop:winnow-beta}
There is an exact, faithful functor
\[
\beta: \cM \to \cM^\win
\]
that preserves weights, and such that the restrictions
\[
\beta: \cM_{n} \to \cM^\win_{n}
\qquad\text{and}\qquad
\beta: \cM_{[n,n+1]} \to \cM^\win_{[n,n+1]}
\]
are equivalences of categories.
\end{prop}

\begin{rmk}\label{rmk:winnow-add}
For $\cF,\cG \in \cM$, the natural map
\[
\beta: \Hom(\cF,\cG) \to \Hom(\beta(\cF), \beta(\cG))
\]
is injective, but not, in general, surjective.  Indeed, if $\cM$ contains any nonsplit exact  short exact sequence $0 \to L' \to X \to L \to 0$ with $\wt(L') \ne \wt(L) - 1$, then~\eqref{eqn:orlov-mixed} (together with~\eqref{eqn:ext1-real}) implies that
\[
0 \to \beta(L') \to \beta(X) \to \beta(L) \to 0
\]
must split.  A morphism $\beta(L) \to \beta(X)$ or $\beta(X) \to \beta(L')$ that splits this sequence cannot come from a morphism in $\cM$: this is the sense in which the winnowing procedure ``adds new morphisms,'' as noted in Section~\ref{sect:intro}.
\end{rmk}

\begin{proof}
The construction we will give is essentially the same as that in the proof of~\cite[Proposition~5.9]{ar}.
	Let $\beta: \cM\to\cM^\win$ be the functor sending $X\in \cM$
	to the complex $(X^\bullet, \delta)$ with $X^i = \gr_{-i}X[i]$, and with differential $\delta^i: X^i \to X^{i+1}$ given by the third morphism in the functorial distinguished triangle
\begin{equation}\label{eqn:funct-triangle}
	\gr_{-i-1}X[i] \to (X_{[-i-1,-i]})[i] \to \gr_{-i}X[i] \overset{\delta^i}{\to}
		\gr_{-i-1}X[i+1]. 
\end{equation}
Here, and below, we write $X_{[n,m]}$ for $(X_{\ge n})_{\le m}$.
	To check that $\delta^{i+1} \circ \delta^i = 0$, consider the compositions
\[
    \xymatrix@R=0pt{
      &\gr_{-i-1} X\ar@{^(->}[dr] \\
      X_{[-i-2,-i-1]} \ar@{->>}[ur] \ar@{^(->}[dr]&  & X_{[-i-1,-i]}\\
      &X_{[-i-2,-i]} \ar@{->>}[ur]}
\]
Let $Y$ be the cone of the map $X_{[-i-2,-i-1]} \to X_{[-i-1,-i]}$ of the above diagram.  The octahedral diagram associated to each separate composition yields the distinguished triangles
   \begin{gather*}
   \gr_{-i-2}X[1] \to Y \to \gr_{-i}X \overset{q}\to \gr_{-i-2}X[2] \\
   \gr_{-i}X \to Y \to \gr_{-i-2}X[1] \to \gr_{-i}X[1].
   \end{gather*}
The second triangle obviously splits, since $\Hom_\cM(\gr_{-i-2}X, \gr_{-i}X) = 0$.  But then the first triangle must also split, so $q = 0$ and thus
 $\delta^{i+1} \circ \delta^i =q[i] = 0$. 

The definition of $\beta$ on morphisms is the same:  for $f: X\to Y$, the morphism $\beta(f)$ is the morphism of complexes such that $\beta(f)^i = \gr_{-i}f[i]$. That we in fact obtain a morphism of complexes follows from the fact that the construction of $\beta(f)$ determines morphisms of triangles of the form \eqref{eqn:funct-triangle}.  As explained in the proof of~\cite[Proposition~5.9]{ar}, the exactness of $\beta$ can be deduced from the exactness of $\gr_{-i}$ combined with~\cite[Lemma~2.5]{ar}.

For $X, Y \in \cM$, there are no nonzero maps $\gr_{-i}X[i] \to \gr_{-i+1}Y[i-1]$, and so no nontrivial homotopies between maps $\beta(X) \to \beta(Y)$.  Thus, for $f: X \to Y$, if $\beta(f) = 0$, it must be that $\gr_{-i}f = 0$ for all $i$, and hence that $f = 0$.  This shows that $\beta$ is faithful.

By~\cite[Corollary~5.5]{ar}, every object in $\cM^\win_{[n,n+1]}$ can be represented by a chain complex $(X^\bullet, d_X)$ where $X^i$ vanishes unless $i = -n,-n-1$.  The differential $\delta^{-n-1} : X^{-n-1} \to X^{-n}$ gives rise to an element
\begin{multline*}
e = \delta^{-n-1}[n+1]\in
\Hom_{\cD}(X^{-n-1}[n+1], X^{-n}[n+1]) \\
\cong \Hom^1_\cD(X^{-n-1}[n+1], X^{-n}[n]) \cong \Ext^1_\cM(X^{-n-1}[n+1], X^{-n}[n]),
\end{multline*}
where the last isomorphism comes from~\eqref{eqn:ext1-real}. The construction of $\cM^\win$ in \cite{ar} implies that $X^{-i}[i]\in\cM$ (in fact in $\cM_i$) for all $i$. Let $\tilde X$ be the middle term of the short exact sequence in $\cM$ determined by $e$.  Tracing through the above construction, one finds that $\beta(\tilde X) \cong (X^\bullet,\delta)$.  Thus, $\beta: \cM_{[n,n+1]} \to \cM^\win_{[n,n+1]}$ is essentially surjective.

To complete the proof, we have left to show that $\beta$ is full on $\cM_{[n,n+1]}$. For pure objects $X_n$ and $Y_n$ of weight $n$, the functor $\beta$ is essentially the identity and clearly gives an isomorphism $\Hom_\cM(X_n,Y_n)\simeq \Hom_{\cM^\win}(\beta(X_n),\beta(Y_n))$. If $Y$ is an arbitrary object in $\cM_{[n,n+1]}$, then $\Hom_{\cM}(X_n,Y)\simeq \Hom_{\cM}(X_n,\gr_nY)$ since $\Hom_{\cM}(X_n,\gr_{n+1}Y)=0$.  When $X_{n+1}$ is a pure object of weight $n+1$ and $Y$ is again an arbitrary object in $\cM_{[n,n+1]}$, there is an exact sequence
\begin{multline*}
    0\to \Hom_{\cM}(X_{n+1},Y) \to \Hom_{\cM}(X_{n+1}, \gr_{n+1} Y) \\
    \to \Ext^1_{\cM}(X_{n+1},\gr_n Y) \to \Ext^1_{\cM}(X_{n+1},Y)\to 0.
\end{multline*}
Let us compare this to the corresponding sequence in $\cM^\win$.
Since $\beta$ is essentially surjective on $\cM_{[n,n+1]}$, we know in particular that every extension of $\beta X_{n+1}$ by $\beta \gr_n Y$ occurs in the essential image of $\beta$.  Thus, the map $\beta: \Ext^1_{\cM}(X_{n+1},\gr_n Y)\to\Ext^1_{\cM^\win}(\beta X_{n+1},\beta\gr_n Y)$ is surjective.  The canonicity of the differential in the definition of $\beta$ implies that it is an injection as well.  That $\beta$ is an isomorphism on the middle two terms in the sequence above implies it must be an isomorphism on the outer terms.  Finally, for $X$ and $Y$ both arbitrary in $\cM_{[n,n+1]}$, we have an exact sequence
\begin{multline*}
    0\to \Hom_\cM(\gr_{n+1}X,Y) \to \Hom_\cM(X, Y) \\ \to \Hom_\cM(\gr_n X,Y) \to \Ext^1_\cM(\gr_{n+1} X,Y)\to \ldots
\end{multline*}
From above, $\beta$ is an isomorphism on $\Hom_\cM(\gr_{n+1}X,Y)$, $\Hom_\cM(\gr_n X,Y)$, and $\Ext^1_\cM(\gr_{n+1} X,Y)$.  Therefore, it must be an isomorphism on $\Hom_\cM(X,Y)$.  
\end{proof}

\section{Preliminaries on perverse sheaves and mixed Hodge modules}
\label{sect:prelim}

Fix, once and for all, a field $\F \subset \R$.  This will be the coefficient field for all constructible sheaves and mixed Hodge modules.  Let $X$ be a smooth variety over $\C$ that is endowed with a fixed algebraic stratification $\cS = \{X_s\}_{s \in S}$.  We write
$j_s: X_s \to X$ for the inclusion of $X_s$ into $X$.  Assume that each stratum $X_s$ is isomorphic to an affine space: $X_s \cong \C^{\dim  X_s}$. (Here, and throughout the paper, the notation ``$\dim$'' when applied to varieties will always denote complex dimension.)  We will impose a stronger condition on the stratification in Section~\ref{subsect:aff-even} below.

\subsection{Perverse $\F$-sheaves}

Let $\dsp(X)$ denote the triangulated category of bounded complexes of $\F$-sheaves on $X$ in the analytic topology that are constructible with respect to $\cS$.  (This category is usually called $\Db_{\mathrm{c}}(X)$ or $\Db_{\mathrm{c},\cS}(X)$, but we use $\dsp(X)$ to forestall confusion with the case of mixed Hodge modules below.)  Let $\Pervs(X) \subset \dsp(X)$ denote the abelian category of perverse $\F$-sheaves that are constructible with respect to $\cS$.  The simple objects in $\Pervs(X)$ are those of the form
\[
\IC_s = j_{s!*}((\text{constant sheaf with value $\F$ on $X_s$})[\dim  X_s]).
\]
The assumption that each $X_s$ is an affine space implies that the realization functor
\begin{equation}\label{eqn:dsp-equiv}
\Db\Pervs(X) \simto \dsp(X).
\end{equation}
is an equivalence of categories~\cite[Corollary~3.3.2]{bgs}.

\subsection{Mixed Hodge $\F$-modules}
\label{subsect:mixedhodge}

Let $\MHM(X)$ denote the category of mixed Hodge $\F$-modules on $X$, and consider its derived category $\Db\MHM(X)$.  Recall that every mixed Hodge module $\cF$ comes with a \emph{weight filtration} as part of its definition, as well as with an underlying perverse sheaf, denoted $\exs(\cF)$.  There are functors obeying the formalism of Grothendieck's ``six operations''~\cite[Theorem~0.1]{smhm}, and their behavior with respect to weights~\cite[p.~225]{smhm} resembles that of mixed $\ell$-adic perverse sheaves.    We refer the reader to \cite{smhm} for a complete introduction to and proofs of facts concerning mixed Hodge modules.

For any smooth complex quasiprojective variety $V$, we write $\uF_V$, or simply $\uF$ if there is no confusion, for the trivial (polarizable) Hodge $\F$-module on $V$.  (We will henceforth omit the word ``polarizable''; all pure Hodge modules or Hodge structures should implicitly be assumed to be polarizable.)  This is a simple object in $\MHM(X)$ of weight $\dim  X$, and its underlying perverse sheaf is a shift (by $\dim  X$) of a constant sheaf.  More generally, for each stratum $X_s$, there is, up to isomorphism, a unique simple object
\[
\MC_s \in \MHM(X)
\qquad\text{such that}\qquad
j_s^*\MC_s \cong \uF_{X_s}.
\]
This object has weight $\dim  X_s$, and its underlying perverse sheaf is $\IC_s$.  Let
\[
\MHMs(X) \subset \MHM(X)
\qquad\text{resp.}\qquad
\dsm(X) \subset \Db\MHM(X)
\]
be the Serre subcategory (resp.~full triangulated subcategory) generated by objects of the form $\MC_s(n)$.  (Here, $\cF \mapsto \cF(1)$ is the Tate twist; $\MC_s(n)$ is a simple object of weight $\dim  X_s - 2n$.)  Note that even on a point endowed with the trivial stratification $\cT$, the category $\MHM_\cT(\pt)$ contains far fewer simple objects than $\MHM(\pt)$. For instance, there exist simple objects of odd weight in $\MHM(\pt)$, whose underlying vector spaces necessarily have dimension greater than $1$.  

The category $\dsm(X)$ can also be described as the full subcategory of $\Db\MHM(X)$ consisting of complexes $\cF$ each of whose cohomology objects $H^i(\cF)$ lie in $\MHMs(X)$.  It follows from the definition of $\dsm(X)$ that for $\cF,\cG \in \MHMs(X)$, we have
\begin{equation}\label{eqn:dsm-hom}
\Hom^i_{\dsm(X)}(\cF,\cG) = \Ext_{\MHM(X)}^i(\cF,\cG).
\end{equation}
The category $\MHMs(X)$ is the heart of the $t$-structure on $\dsm(X)$ induced by the standard $t$-structure on $\Db\MHM(X)$.  (We will simply call it ``the standard $t$-structure on $\dsm(X)$.'')  Together, $\MHMs(X)$ and $\dsm(X)$ form a mixed triangulated category in the sense of Section~\ref{subsect:mixed-kos}.  For brevity, the category of~\eqref{eqn:pure-defn} will usually be denoted
\[
\Pure(X) = \Pure(\dsm(X)).
\]
Finally, the underlying-perverse-sheaf functor $\cF \mapsto \exs(\cF)$ preserves the categories we have defined associated to the stratification $\cS$, so it induces functors
\[
\exs: \MHMs(X) \to \Pervs(X)
\qquad\text{and}\qquad
\exs: \dsm(X) \to \dsp(X).
\]
This functor is $t$-exact, and an object $\cF \in \dsm(X)$ belongs to $\MHMs(X)$ if and only if $\exs(\cF) \in \Pervs(X)$.

\subsection{Hom-groups}

Let $a: X \to \pt$ denote the constant map.  Given $\cF,\cG \in \dsm(X)$ and $i \in \Z$, we define a mixed Hodge structure $\uHom^i(\cF,\cG)$ by
\[
\uHom^i(\cF,\cG) = H^i(a_*\mathscr{H}\mathit{om}(\cF,\cG)).
\]
In other words, $\uHom^i(\cF,\cG)$ is an object of $\MHM(\pt)$ that (because $\exs$ commutes with all sheaf functors) is equipped with a natural isomorphism
\[
\exs\uHom^i(\cF,\cG) \cong \Hom^i(\exs\cF,\exs\cG).
\]
The following natural short exact sequence expresses the relationship between $\Hom$-groups in $\dsm(X)$ and those in $\dsp(X)$~\cite[Theorem~2.10]{semhm}:
\begin{equation}\label{eqn:hodge-ses}
0 \to \HH^1(\uHom^{i-1}(\cF,\cG)) \to \Hom^i(\cF,\cG) \to \HH^0(\uHom^i(\cF,\cG)) \to 0.
\end{equation}
Here, the functor $\HH^j = \Ext^j_{\MHM(\pt)}(\uF_{\pt},{-}): \MHM(\pt) \to \Vect_\F$ is the Hodge cohomology functor.  All $\Hom$-groups in $\dsp(X)$ are finite-dimensional, but it should be noted that $\Hom$-groups in $\dsm(X)$ (or even in $\Db\MHM(\pt)$) may have infinite dimension. For example, we have $\Ext^1_{\MHM(\pt)}(\uF, \uF(n)) \cong \C/(2\pi i)^n\F$, which is only finite-dimensional if $\F=\R$.

\subsection{Affine even stratifications}
\label{subsect:aff-even}

The main results of this paper hold for stratified varieties that satisfy the following conditions:

\begin{defn}\label{defn:aff-even}
A stratification $\cS = \{X_s\}_{s \in S}$ of a variety $X$ is called an \emph{affine even stratification} if the following two conditions hold:
\begin{enumerate}
\item Each $X_s$ is isomorphic to the affine space $\C^{\dim  X_s}$.
\item For all $s, t \in S$ and $i \in \Z$, the mixed Hodge module $H^i(j_t^*\MC_s)$ vanishes if $i \not \equiv \dim  X_s - \dim  X_t \pmod 2$, and is isomorphic to a direct sum of copies of $\uF_{X_t}((\dim  X_t -\dim  X_s-i)/2)$ otherwise.\label{it:parity}
\end{enumerate}
\end{defn}
\noindent
Note in particular that condition~\eqref{it:parity} above implies that $j_t^*\MC_s$ is pure, and hence semisimple.

\begin{lem}\label{lem:aff-even-crit}
Suppose that $X$ has a stratification $\cS = \{X_s\}_{s \in S}$ by affine spaces.  Assume that for each stratum $X_s \subset X$, there is a proper morphism $\pi_s: Y_s \to \overline{X_s}$ such that the following conditions hold:
\begin{enumerate}
\item $Y_s$ is smooth.
\item The restriction $\pi_s: \pi_s^{-1}(X_s) \to X_s$ is an isomorphism.\label{it:ressing}
\item For any $X_t \subset \overline{X_s}$, the projection $\pi_s: \pi_s^{-1}(X_t) \to X_t$ is a trivial fibration, and $\pi_s^{-1}(X_t)$ has an affine paving.
\end{enumerate}
Then $\cS$ is an affine even stratification.
\end{lem}
\begin{proof}
The assumptions imply that $\pi_s$ is surjective. Since $\pi_s$ is proper, the object $\pi_{s*}\uF_{Y_s}$ is pure (of weight $\dim Y_s = \dim  X_s$) and therefore semisimple.  It is clear from condition~\eqref{it:ressing} above that $\MC_s$ occurs as a direct summand of $\pi_{s*}\uF_{Y_s}$.  Now, choose a stratum $X_t \subset \overline{X_s}$, and let $Z = \pi_s^{-1}(X_t)$.  To prove condition~\eqref{it:parity} of Definition~\ref{defn:aff-even}, it suffices to prove the following claim: The object $\cF = j_t^*\pi_{s*}\uF_{Y_s} \cong \pi_{s*} \uF_Z[\dim  X_s - \dim Z]$ has the property that
\begin{equation}\label{eqn:affeven-cohom}
H^i(\cF) \cong
\left\{
\begin{array}{@{}l}
\text{$0$ \quad if $i \not\equiv \dim  X_s - \dim  X_t \pmod 2$,} \\
\text{a direct sum of copies of $\uF_{X_t}(\frac{\dim  X_t - \dim  X_s - i}{2})$ otherwise.}
\end{array}
\right.
\end{equation}

To prove this claim, consider the full subcategory $\MHM_\const(X_t) \subset \MHM(X_t)$ consisting of objects whose underlying perverse sheaf is constant.  Let $\Db_{\cS,\const}(X_t)$ be the corresponding full triangulated subcategory of $\Db\MHM(X_t)$, and let $a: X_t \to \pt$ be the constant map.  Considering the cohomology of a constant sheaf on $X_t$, one sees that the functor $a_*[-\dim  X_t]: \Db_{\cS,\const}(X_t) \to \Db\MHM(\pt)$ is $t$-exact, preserves weights, and kills no nonzero object of $\MHM_\const(X_t)$.  Now, let $\cG \in \MHM_\const(X_t)$ be a simple object that is \emph{not} isomorphic to $\uF_{X_t}(n)$ for any $n$.  Since $\uF_{X_t}(n) \cong a^*\uF_\pt[\dim  X_t](n)$, we have
\[
\Hom_{\MHM(X_t)}(\uF_{X_t}(n), \cG) \cong \Hom_{\MHM(\pt)}(\uF_\pt(n), a_*\cG[-\dim  X_t]) = 0.
\]
That is, $a_*\cG[-\dim  X_t]$ is a nonzero object of $\MHM(\pt)$ containing no subobject isomorphic to $ \uF_\pt(n)$.

Since $\pi_t$ is a trivial fibration, the underlying perverse sheaves of the $H^i(\cF)$ are constant, so $\cF \in \Db_{\cS,\const}(X_t)$.  Since $\pi_s: Z \to X_t$ is proper and $\uF_Z[\dim  X_s - \dim Z]$ has weights${}\le \dim  X_s$, it follows that $\cF$ has weights${}\le \dim  X_s$ as well.  In fact, $\cF$ must be pure of weight $\dim  X_s$: if not, $a_*\cF$ would not be pure either, but an easy induction on the number of cells in the affine paving of $Z$ shows that
\begin{equation}\label{eqn:haf-calc}
H^i(a_*\pi_{s*}\uF_Z) \cong
\left\{
\begin{array}{@{}l}
\text{$0$ \quad if $i \not\equiv \dim Z \pmod 2$,} \\
\text{a direct sum of copies of $\uF_{\pt}(\frac{-\dim Z - i}{2})$ otherwise.}
\end{array}
\right.
\end{equation}
Since $\cF$ is pure, it is semisimple.  If $\cF$ did not have the property~\eqref{eqn:affeven-cohom}, then it would contain a simple summand $\cG$ as in the preceding paragraph, and then $a_*\cF$ would contain a summand containing no subobject isomorphic to $\uF_\pt(n)$,  contradicting~\eqref{eqn:haf-calc}.  Thus, the claim~\eqref{eqn:affeven-cohom} holds, as desired.
\end{proof}

For another perspective on the preceding argument, one can use the rigidity theorem for variations of Hodge structures~\cite[Theorem~7.24]{schmid} to show that $a_*[-\dim X_t]: \Db_{\cS,\const}(X_t) \to \Db\MHM(\pt)$ is actually an equivalence of categories.  This makes it clear that conditions~\eqref{eqn:affeven-cohom} and~\eqref{eqn:haf-calc} are equivalent.

\begin{cor}
The stratification of any partial flag variety of a reductive algebraic group by Schubert cells is an affine even stratification.
\end{cor}
\begin{proof}
It follows from~\cite[Theorem~2]{gau:fbsr} that the Bott--Samelson--Demazure resolution of a Schubert variety satisfies the conditions of Lemma~\ref{lem:aff-even-crit}.  (The last sentence of~\cite[Theorem~2]{gau:fbsr} is false, cf.~\cite{gau:efbsr}, but the parts we need are correct.)
\end{proof}

\begin{lem}\label{lem:aff-even-std}
Suppose $X$ has an affine even stratification $\cS = \{X_s\}_{s \in S}$.  For each $s \in S$, the object $\Delta_s = j_{s!}\uF_{X_s}$ lies in $\MHMs(X)$, and if $\cF \in \Db\MHM(X)$ has the property that $j_t^*\cF \in \dsm(X_t)$ for all $t \in S$, then $\cF \in \dsm(X)$. 

Dually, the object $\nabla_s = j_{s*}\uF_{X_s}$ lies in $\MHMs(X)$, and if $j_t^!\cF \in \dsm(X_t)$ for all $t \in S$, then $\cF \in \dsm(X)$.  
\end{lem}
\begin{proof}
We will treat only the first part of the lemma; the second part is similar.  We proceed by induction on the number of strata in $X$.  If $X$ consists of a single stratum, there is nothing to prove.  Otherwise, let $X_u \subset X$ be an open stratum, and let $i: Z \hookrightarrow X$ be the inclusion of its complement.  In $\Db\MHM(X)$, there is a distinguished triangle
\[
\Delta_u \to \MC_u \to i_*i^*\MC_u \to.
\]
It follows from part~\eqref{it:parity} of Definition~\ref{defn:aff-even} that $i^*\MC_u \in \Db\MHM(Z)$ satisfies the assumptions of the lemma.  Since $Z$ has fewer strata, we have by induction that $i^*\MC_u \in \dsm(Z)$.  It follows that $i_*i^*\MC_u \in \dsm(X)$, so $\Delta_u \in \dsm(X)$.  Since the inclusion $j_u: X_u \to X$ is affine, the object $\exs(\Delta_u) \in \dsp(X)$ is a perverse sheaf by~\cite[Corollaire~4.1.3]{bbd}, so $\Delta_u \in \MHMs(X)$.

Now, for general $\cF \in \Db\MHM(X)$, form the distinguished triangle
\[
j_{u!}j_u^!\cF \to \cF \to i_*i^*\cF \to.
\]
The first term lies in the subcategory generated by Tate twists of $\Delta_u$, and the last term again lies in $\dsm(X)$ by induction, so $\cF$ lies in $\dsm(X)$ as well.
\end{proof}

\begin{rmk}\label{rmk:failure}
It is easy to see that $\MHMs(X)$ is never Koszul: there exist simple objects $\cF$ and $\cG$ such that
\begin{equation}\label{eqn:nkos-mhm}
\Ext^1_{\MHMs(X)}(\cF,\cG) \ne 0 \qquad \text{with $\wt(\cG) < \wt(\cF) - 1$.}
\end{equation}
For instance, one can take $\cF = \MC_s$ and $\cG = \MC_s(n)$ for some $n \ge 1$; the claim above follows from~\eqref{eqn:hodge-ses} and the fact that $\Ext^1_{\MHM(\pt)}(\uF, \uF(n)) \cong \C/(2\pi i)^n\F$.

It is likewise true that the category of mixed $\ell$-adic perverse sheaves on a variety over a finite field is never Koszul: one can find simple $\cF$ and $\cG$ such that
\begin{equation}\label{eqn:nkos-perv}
\Ext^1_{\Pervs^{\mathrm{m}}(X;\bar{\mathbb{Q}}_\ell)}(\cF,\cG) \ne 0 \qquad \text{with $\wt(\cG) > \wt(\cF) - 1$.}
\end{equation}
Indeed, by~\cite[(5.1.2.5)]{bbd}, this happens for $\cF = \cG = \IC_s$.  The fact that the failure of Koszulity happens in opposite ways in~\eqref{eqn:nkos-mhm} and~\eqref{eqn:nkos-perv} is what necessitates the use of two different approaches to achieve Koszulity in the two settings.
\end{rmk}

\subsection{$\exs$-projective mixed Hodge modules}

A key property of spaces with affine even stratifications is that the category $\Pervs(X)$ has enough projectives and injectives~\cite[Theorem~3.3.1]{bgs}.  Unfortunately, $\MHMs(X)$ does not have enough projectives or injectives, even on a point.  The following notion serves as a substitute.

\begin{defn}
Suppose $X$ has an affine even stratification $\cS = \{X_s\}_{s \in S}$.  An object $\cF \in \MHMs(X)$ is said to be \emph{$\exs$-projective} (resp.~\emph{$\exs$-injective}) if $\exs(\cF)$ is a projective (resp.~injective) object of $\Pervs(X)$.
\end{defn}

In~\cite{schnuerer}, these objects were called ``perverse-projective'' and ``perverse-injective.''  We emphasize that a $\exs$-projective object in $\MHMs(X)$ is \emph{not} projective.

\begin{prop}[{\cite[Proposition~14 and Corollary~15]{schnuerer}}]\label{prop:pseudoproj}
Suppose $X$ has an affine even stratification $\cS = \{X_s\}_{s \in S}$.  The category $\MHMs(X)$ has enough $\exs$-projective objects, and every object admits a $\exs$-projective resolution of finite length.\qed
\end{prop}

There is an analogous fact for $\exs$-injective objects.  Proposition \ref{prop:pseudoproj} is stated in \cite{schnuerer} only for real coefficients, but it is easy to see that the coefficient field plays no role in the proof, and that the result is in fact valid for arbitrary $\F$.

\begin{lem}
If $\cP \in \MHMs(X)$ is $\exs$-projective, then for any $\cG \in \MHMs(X)$, we have $\Hom^i(\cP,\cG) = 0$ for all $i \ge 2$.
\end{lem}
\begin{proof}
Since $\exs(\cP)$ is projective, $\uHom^i(\cP,\cG) = 0$ for all $i \ge 1$.  The result then follows from~\eqref{eqn:hodge-ses}.
\end{proof}

\begin{prop}\label{prop:rsc-equiv}
The realization functor
\[
\rsc: \Db\MHMs(X) \to \dsm(X)
\]
is an equivalence of categories.
\end{prop}
\begin{proof}
Let $\cF, \cG \in \MHMs(X)$, and let $f: \cF \to \cG[n]$ be a morphism in $\dsm(X)$ with $n > 0$.  According to~\cite[Proposition~3.1.16]{bbd}, it suffices to establish the following claim: There is a surjection $p: \tilde \cF \to \cF$ in $\MHMs(X)$ such that $f \circ p = 0$.  If $n = 1$, this is guaranteed by~\cite[Remarque~3.1.17(ii)]{bbd}; if $n \ge 2$, the preceding lemma implies that any $\exs$-projective cover of $\cF$ has the required property.
\end{proof}

An immediate corollary is that the natural transformation~\eqref{eqn:exti-real}, which can be rewritten using~\eqref{eqn:dsm-hom} as
\begin{equation}\label{eqn:rsc-equiv-transf}
\rsc: \Ext^i_{\MHMs(X)}(\cF,\cG) \simto \Ext^i_{\MHM(X)}(\cF,\cG),
\end{equation}
is an isomorphism for all $\cF, \cG \in \MHMs(X)$.

\section{Lemmas on Hom- and Ext-groups}
\label{sect:homext}

For the rest of the paper, we consider only affine even stratifications.  In this section, we prove a number of technical results about $\uHom^i({-},{-})$, $\Ext^i_{\MHMs(X)}({-},{-})$, and a new functor denoted $\gHom^i({-},{-})$.

\subsection{Hom-groups for pure objects}
\label{subsect:hompure}

The following lemma and its corollary will be used many times in this section and the next.  For now, the main consequence is Proposition~\ref{prop:dimhomfin}, establishing condition~\eqref{eqn:dimhomfin} for $\dsm(X)$. 

\begin{lem}\label{lem:hom-semis}
Suppose $\cF,\cG \in \MHMs(X)$ are simple.  The mixed Hodge structure
\[
\uHom^i(\cF,\cG)
\]
vanishes if $i \not\equiv \wt(\cF) - \wt(\cG) \pmod 2$, and is isomorphic to a direct sum of finitely many copies of $\uF((\wt(\cF) - \wt(\cG) -i)/2)$ otherwise.
\end{lem}
\begin{proof}
Let us assume that $\cF = \MC_s$ and $\cG = \MC_t$.  Since every simple object is isomorphic to a Tate twist of some $\MC_u$, it suffices to treat this special case.

We proceed by induction on the number of strata in $X$.  Let $j_u: X_u \to X$ be the inclusion of a closed stratum, and let $h: U \to X$ be the inclusion of the complementary open subset.  Applying $\uHom^\bullet(\MC_s, {-})$ to the distinguished triangle $j_{u*}j_u^!\MC_t \to \MC_t \to h_*h^*\MC_t \to$, we obtain a long exact sequence
\begin{equation}\label{eqn:hom-semis-les}
\cdots \to \uHom^i(j_u^*\MC_s, j_u^!\MC_t) \to \uHom^i(\MC_s, \MC_t) \to \uHom^i(h^*\MC_s,h^*\MC_t) \to \cdots.
\end{equation}
Now, $j_u^*\MC_s$ is a direct sum of finitely many objects of the form
\[
\uF_{X_u}((\dim X_u - \dim X_s - k)/2)[-k] \qquad\text{where $k \equiv \dim X_s - \dim X_u \pmod 2$.}
\]
Similarly, $j_u^!\MC_t$ is a direct sum of finitely many objects of the form
\[
\uF_{X_u}((\dim X_u - \dim X_t - l)/2)[-l] \qquad\text{where $l \equiv \dim X_t - \dim X_u \pmod 2$.}
\]
Recall that $\uHom(\uF_{X_u}, \uF_{X_u}) \cong \uF_\pt$, and that $\uHom^i(\uF_{X_u}, \uF_{X_u}) = 0$ for $i \ne 0$. Thus,
\begin{multline*}
\uHom^i(\uF_{X_u}((\dim X_u - \dim X_s - k)/2)[-k], \uF_{X_u}((\dim X_u - \dim X_t - l)/2)[-l]) \\
\cong
\begin{cases}
\uF_\pt((\dim X_s - \dim X_t + k - l)/2)& \text{if $i = l - k$,} \\
0 & \text{otherwise.}
\end{cases}
\end{multline*}
It follows that $\uHom^i(j_u^*\MC_s, j_u^!\MC_t)$ vanishes when $i \not\equiv \dim X_s - \dim X_t \pmod 2$, and is a direct sum of finitely many copies of $\underline{\F}((\dim X_s - \dim X_t - i)/2)$ otherwise.  By induction, the same description holds for $\uHom^i(h^*\MC_s, h^*\MC_t)$, and then the proposition follows from the long exact sequence~\eqref{eqn:hom-semis-les}.
\end{proof}

\begin{cor}\label{cor:hh-perv}
Suppose $\cF,\cG \in \MHMs(X)$ are simple. If $n = (\wt(\cG) - \wt(\cF) + i)/2$, then the natural maps
\begin{equation}\label{eqn:hh-perv-iso}
\Ext^i_{\MHM(X)}(\cF,\cG(n)) \simto \HH^0(\uHom^i(\cF,\cG(n))) \simto \Ext^i_{\Pervs(X)}(\exs\cF,\exs\cG)
\end{equation}
are both isomorphisms.  On the other hand, if $n \ne (\wt(\cG) - \wt(\cF) + i )/2$, then
\begin{equation}\label{eqn:hh-perv-van}
\HH^0(\uHom^i(\cF,\cG(n)) = 0.
\end{equation}
As a consequence, for all $i \ge 0$, the functor $\exs$ induces an isomorphism
\begin{equation}\label{eqn:hh-perv-oplus}
\bigoplus_{n \in \Z} \HH^0(\uHom^i(\cF, \cG(n))) \simto \Ext^i_{\Pervs(X)}(\exs\cF, \exs\cG).
\end{equation}
\end{cor}
\begin{proof}
If $n = (\wt(\cG) - \wt(\cF) +i)/2$, then, by Lemma~\ref{lem:hom-semis}, $\uHom^i(\cF,\cG(n))$ is isomorphic to a direct sum of copies of $\uF_\pt$.  For such an object, the functor $\HH^0({-})$ coincides with simply taking the underlying vector space; this gives the second isomorphism in~\eqref{eqn:hh-perv-iso}.  That lemma also tells us that $\uHom^{i-1}(\cF,\cG(n)) = 0$, so the first isomorphism in~\eqref{eqn:hh-perv-iso} follows from~\eqref{eqn:hodge-ses}.

For~\eqref{eqn:hh-perv-van}, we have that $\uHom^i(\cF,\cG(n))$ is pure with nonzero weight.

When $i \equiv \wt(\cG) - \wt(\cF) \pmod 2$,~\eqref{eqn:hh-perv-oplus} is an immediate consequence of~\eqref{eqn:hh-perv-iso} and~\eqref{eqn:hh-perv-van}.  When $i \not\equiv \wt(\cG) - \wt(\cF) \pmod 2$, the left-hand side of~\eqref{eqn:hh-perv-oplus} vanishes by~\eqref{eqn:hh-perv-van}.  The right-hand side also vanishes: we have $\Ext^i(\exs\cF,\exs\cG) \cong \exs \uHom^i(\cF,\cG)$, and $\uHom^i(\cF,\cG)$ vanishes by Lemma~\ref{lem:hom-semis}.  Thus,~\eqref{eqn:hh-perv-oplus} holds in all cases.
\end{proof}

\begin{prop}\label{prop:dimhomfin}
If $\cF,\cG \in \Pure(X)$, then $\dim \Hom_{\dsm(X)}(\cF,\cG) < \infty$.
\end{prop}
\begin{proof}
Recall from Section~\ref{subsect:winnowmix} that every object of $\Pure(X)$ is semisimple.  Thus, it suffices to consider the special case where $\cF = \cF_0[-\wt(\cF_0)]$ and $\cG = \cG_0[-\wt(\cG_0)]$, where $\cF_0$ and $\cG_0$ are pure objects of $\MHMs(X)$.  Then
\[
\Hom(\cF,\cG) \cong \Ext^{\wt(\cF_0) - \wt(\cG_0)}_{\MHM(X)}(\cF_0,\cG_0),
\]
and the latter is finite-dimensional by~\eqref{eqn:hh-perv-iso}.
\end{proof}

\subsection{A new delta functor for $\MHMs(X)$}
\label{subsect:new-delta}

For $\cF,\cG \in \MHMs(X)$, define
\[
\gHom^i(\cF,\cG) = \exs\gr_0 \uHom^i(\cF,\cG).
\]
Since $\gr_0$ and $\exs$ are both exact, this clearly defines a $\delta$-functor on $\MHMs(X)$.  The mixed Hodge structure $\gr_0 \uHom^i(\cF,\cG)$ is a direct sum of copies of $\uF$: this follows from the more general observation (implied by Lemma~\ref{lem:hom-semis}) that every composition factor of $\uHom^i(\cF,\cG)$ is of the form $\uF(n)$ for some $n$.

\begin{lem}\label{lem:ghom-basic}
Let $\cF, \cG \in \MHMs(X)$.
\begin{enumerate}
\item If $\cF$ has weights${}\le n$ and $\cG$ has weights${}>n-i$, then $\gHom^i(\cF,\cG) = 0$.\label{it:ghom-mixed}
\item If $\cF$ has weights${}\ge n$ and $\cG$ has weights${}<n-i$, then $\gHom^i(\cF,\cG) = 0$.\label{it:ghom-comixed}
\item For any $n \in \Z$, there is a natural transformation
\[
\gHom(\cF,\cG) \to \gHom(\cF_{\le n}, \cG_{\le n})
\]
that is the identity map if $\cF$ and $\cG$ happen to have weights${}\le n$.\label{it:ghom-leq}
\item For any $n \in \Z$, there is a natural transformation
\[
\gHom(\cF,\cG) \to \gHom(\cF_{\ge n}, \cG_{\ge n})
\]
that is the identity map if $\cF$ and $\cG$ happen to have weights${}\ge n$.\label{it:ghom-geq}
\item There is a natural morphism of $\delta$-functors
\[
\hexs: \Ext^i_{\MHM(X)}(\cF,\cG) \to \gHom^i(\cF,\cG).
\]
If $\cF$ is pure of weight $n$ and $\cG$ is pure of weight $n-i$, this is an isomorphism.\label{it:ghom-hexs}
\end{enumerate}
\end{lem}
\begin{proof}
For parts~\eqref{it:ghom-mixed} and~\eqref{it:ghom-comixed}, it suffices to treat the case where $\cF$ and $\cG$ are both simple.  In this case, using the assumptions on the weights of $\cF$ and $\cG$, Lemma~\ref{lem:hom-semis} tells us that $\uHom^i(\cF,\cG)$ is pure with nonzero weight.  Thus, $\gr_0 \uHom^i(\cF,\cG) = 0$.

For part~\eqref{it:ghom-leq}, consider the exact sequence
\[
0 \to \gHom(\cF_{\le n},\cG_{\le n}) \to \gHom(\cF_{\le n},\cG) \to \gHom(\cF_{\le n},\cG_{\ge n+1}) \to \cdots.
\]
The last term vanishes by part~\eqref{it:ghom-mixed}, so the first is an isomorphism.  Composing the inverse of that map with the natural map $\gHom(\cF,\cG) \to \gHom(\cF_{\le n},\cG)$ yields the desired natural transformation.  The proof of part~\eqref{it:ghom-geq} is similar.

Lastly, we turn to part~\eqref{it:ghom-hexs}.  Note first that the natural map
\begin{equation}\label{eqn:ghom-hexs1}
\HH^0(\uHom^i(\cF,\cG)_{\le 0}) \simto \HH^0(\uHom^i(\cF,\cG))
\end{equation}
is an isomorphism, as we clearly have $\HH^0(\uHom^i(\cF,\cG)_{> 0}) = 0$.  Next, we have a natural map
\begin{equation}\label{eqn:ghom-hexs2}
\HH^0(\uHom^i(\cF,\cG)_{\le 0}) \to \HH^0(\gr_0 \uHom^i(\cF,\cG)).
\end{equation}
Composing the second map in~\eqref{eqn:hodge-ses} with the inverse of~\eqref{eqn:ghom-hexs1} and then with~\eqref{eqn:ghom-hexs2}, we obtain a natural map $\Ext^i_{\MHM(X)}(\cF,\cG) \to \HH^0(\gr_0 \uHom^i(\cF,\cG))$.  Finally, observe that if $M \in \MHM(\pt)$ is a direct sum of copies of $\uF$, then there is a natural isomorphism $\HH^0(M) \cong \exs(M)$.  In particular, this applies to $M = \gr_0 \uHom^i(\cF,\cG)$, and we thus obtain the desired natural transformation $\hexs$.

If $\cF$ is pure of weight $n$ and $\cG$ is pure of weight $n-i$, then $\uHom^i(\cF,\cG)$ is pure of weight $0$ by Lemma~\ref{lem:hom-semis}, so~\eqref{eqn:ghom-hexs2} is an isomorphism.  The second map in~\eqref{eqn:hodge-ses} is also an isomorphism (see~\eqref{eqn:hh-perv-iso}), so $\hexs$ is an isomorphism in this case.
\end{proof}

\begin{lem}\label{lem:ghom-univ}
The $\delta$-functor $\gHom(-,-)$ is universal.
\end{lem}
\begin{proof}
If $\cF$ is $\exs$-projective, then $\uHom^i(\cF,-)$ vanishes for $i > 0$, and hence so does $\gHom^i(\cF,-)$.  Recall from Proposition~\ref{prop:pseudoproj} that every object is a quotient of a $\exs$-projective one.  It follows that $\gHom^\bullet(-,-)$ is effaceable.
\end{proof}

\section{Winnowed mixed Hodge modules and Koszulity}
\label{sect:koszul}

In Proposition~\ref{prop:dimhomfin}, we established that condition~\eqref{eqn:dimhomfin} holds for the mixed triangulated category $\dsm(X)$.  Condition~\eqref{eqn:mixedsplit} for $\MHMs(X)$ is clear, so the machinery of Section~\ref{subsect:winnowmix} applies, and we have an abelian category
\[
\MHMso(X) \subset \Kb\Pure(X).
\]
In this section, we will prove that $\MHMso(X)$ is a Koszul category.  The main tool will be the $\delta$-functor $\gHom({-},{-})$ that was introduced in the previous section: it will serve as a sort of intermediary between $\Hom$-groups in $\dsm(X)$ (where geometric methods relying on the properties of affine even stratifications are available) and $\Ext$-groups in $\MHMso(X)$ (which are the ones we seek to understand).

Recall from Proposition~\ref{prop:rsc-equiv} that $\rsc$ denotes the realization functor
\[
\rsc: \Db\MHMs(X) \simto \dsm(X).
\]
We now also consider a realization functor for $\MHMso(X)$, which will be denoted
\[
\rso: \Db\MHMso(X) \to \Kb\Pure(X).
\]
Like $\rsc$, this gives rise to a morphism of $\delta$-functors~\eqref{eqn:exti-real}.  These morphisms, along with the morphism $\hexs$ of Lemma~\ref{lem:ghom-basic}\eqref{it:ghom-hexs} and the morphisms of $\Ext$-groups induced by the exact functors $\beta$ and $\exs$, are shown in Figure~\ref{fig:delta}.  That figure also contains a new morphism $\hbeta$ that has not previously been discussed.

\begin{figure}
\[
\xymatrix@C=0pt{
\Ext^i_{\MHMs(X)}({-},{-}) \ar[dr]^{\rsc}_{\sim} \ar[ddd]_{\beta} 
  &&\quad& \Hom^i_{\ds(X)}(\exs(-),\exs(-)) \ar@{=}[d] \\
& \Ext^i_{\MHM(X)}({-},{-}) \ar[rr]^-{\exs} \ar[d]^{\hexs}
  && \Ext^i_{\Pervs(X)}(\exs(-),\exs(-)) \\
& \gHom^i({-},{-}) \ar[dl]^{\hbeta} \\
\Ext^i_{\MHMso(X)}(\beta(-),\beta(-)) \ar[d]_{\rso} \\
\Hom^i_{\Kb\Pure(X)}(\beta(-),\beta(-))}
\]
\caption{$\delta$-functors for $\MHMs(X)$.}\label{fig:delta}
\end{figure}

\begin{lem}\label{lem:hbeta}
There is a morphism of $\delta$-functors
\[
\hbeta: \gHom^\bullet(-,-) \to \Ext_{\MHMso(X)}^\bullet(\beta(-),\beta(-))
\]
that makes the diagram in Figure~\ref{fig:delta} commute.
\end{lem}
\begin{proof}
It follows from parts~\eqref{it:ghom-leq} and~\eqref{it:ghom-geq} of Lemma~\ref{lem:ghom-basic} that for any $n \in \Z$, the functor $\gr_n$ induces a natural transformation
\[
\gr_n: \gHom(\cF,\cG) \to \gHom(\gr_n\cF,\gr_n\cG) \cong \Hom_{\MHMs(X)}(\gr_n\cF,\gr_n\cG),
\]
where the last isomorphism comes from Lemma~\ref{lem:ghom-basic}\eqref{it:ghom-hexs}.  Given $f \in \gHom(\cF,\cG)$, let $\hbeta(f)^n$ denote the map $\gr_{-n}f[n]: \gr_{-n}\cF[n] \to \gr_{-n}\cG[n]$, thought of as a morphism in $\Pure(X)$.  The same construction appeared in the description of the functor $\beta$ in Proposition~\ref{prop:winnow-beta}, and the reasoning given there shows again that, taken together, the $\hbeta(f)^n$ constitute a chain map $\beta(\cF) \to \beta(\cG)$.  It is obvious by construction that the diagram in Figure~\ref{fig:lex} of left-exact functors of abelian categories commutes.

Because $\gHom^\bullet(-,-)$ is universal, the morphism $\hbeta$ extends in a unique way to a morphism of $\delta$-functors.  Next, $\hbeta \circ \hexs \circ \rsc$ and $\beta$ are both morphisms of $\delta$-functors $\Ext_{\MHMs(X)}^\bullet(-,-) \to \Ext_{\MHMso(X)}^\bullet(\beta({-}),\beta({-}))$.  Since they agree on the $0$th terms, and since $\Ext_{\MHMs(X)}^\bullet(-,-)$ is universal, we have $\hbeta \circ \hexs \circ \rsc = \beta$, as desired.
\end{proof}

\begin{figure}
\[
\xymatrix@C=0pt{
\Hom_{\MHMs(X)}({-},{-}) \ar@{=}[dr]^{\rsc} \ar[ddd]_{\beta} \\
& \Hom_{\MHM(X)}({-},{-}) \ar[d]^{\hexs} \\
& \gHom({-},{-}) \ar[dl]^{\hbeta} \\
\Hom_{\MHMso(X)}(\beta(-),\beta(-))}
\]
\caption{Some left-exact functors on $\MHMs(X)$}\label{fig:lex}
\end{figure}

For convenience, let us restate Lemma~\ref{lem:extn-surj} for $\MHMs(X)$, using~\eqref{eqn:rsc-equiv-transf}.

\begin{lem}\label{lem:extdsm-surj}
Let $\cF, \cG \in \MHMs(X)$ be pure objects of weights $n$ and $n-i$, respectively.  The product map
\[
\bigoplus_{\substack{\cK \in \Irr(\MHMs(X))\\ \wt(\cK) = n - 1}}
\Ext_{\MHM(X)}^1(\cF, \cK) \otimes \Ext_{\MHM(X)}^{i-1}(\cK, \cG) \to \Ext_{\MHM(X)}^i(\cF, \cG).
\]
is surjective. \qed
\end{lem}

\begin{lem}\label{lem:ghom-surj}
Let $\cF, \cG \in \MHMs(X)$ be pure objects of weights $n$ and $n-i$, respectively.  The product map
\[
\bigoplus_{\substack{\cK \in \Irr(\MHMs(X))\\ \wt(\cK) = n - 1}}
\gHom^1(\cF, \cK) \otimes \gHom^{i-1}(\cK, \cG) \to \gHom^i(\cF, \cG).
\]
is surjective.
\end{lem}
\begin{proof}
Under the assumptions on the weights of the objects appearing in this statement, we know from Lemma~\ref{lem:ghom-basic}\eqref{it:ghom-hexs} that the map $\hexs$ is an isomorphism.  So this statement follows immediately from Lemma~\ref{lem:extdsm-surj}.
\end{proof}

\begin{lem}\label{lem:homkb-surj}
Let $\cF, \cG \in \MHMso(X)$ be pure objects of weights $n$ and $n-i$, respectively.  The product map
\[
\bigoplus_{\substack{\cK \in \Irr(\MHMso(X))\\ \wt(\cK) = n - 1}}
\Hom_{\Kb\Pure(X)}^1(\cF, \cK) \otimes \Hom_{\Kb\Pure(X)}^{i-1}(\cK, \cG) \to \Hom_{\Kb\Pure(X)}^i(\cF, \cG)
\]
is surjective.
\end{lem}
\begin{proof}
The assumptions on the weights of the objects above mean that the $\Hom$-groups above can be computed in (a shift of) the additive category $\Pure(X)$.  Via the inclusion $\Pure(X) \to \dsm(X)$, these $\Hom$-groups can be identified with $\Ext$-groups over $\MHM(X)$, so this result follows from Lemma~\ref{lem:extdsm-surj}.
\end{proof}

\begin{lem}\label{lem:rso-hbeta-isom}
For all $\cF,\cG \in \MHMs(X)$ and all $i \ge 0$, the map
\begin{equation}\label{eqn:rso-hbeta-isom-comp}
\rso\circ\hbeta: \gHom^i(\cF,\cG) \to \Hom^i_{\Kb\Pure(X)}(\beta(\cF),\beta(\cG))
\end{equation}
is an isomorphism.
\end{lem}
\begin{proof}
By a standard d\'evissage argument, it suffices to prove this when $\cF$ and $\cG$ are simple, so let us assume that that is the case.  If $\wt(\cG) \ne \wt(\cF) -i$, then both groups above vanish (by Lemma~\ref{lem:ghom-basic} and~\eqref{eqn:orlov-mixed}, respectively), so the map is trivially an isomorphism.  We henceforth assume that $\wt(\cG) = \wt(\cF) - i$.

Assume first that $i = 0$ or $i = 1$.  Let $m = \wt(\cG)$.  Then $\cF$ and $\cG$ both belong to $\MHMs(X)_{[m,m+1]}$.  We can compute $\Ext^i_{\MHMs(X)}(\cF,\cG)$ inside the Serre subcategory $\MHMs(X)_{[m,m+1]}$, and likewise for $\Ext^i_{\MHMso(X)}(\beta\cF,\beta\cG)$ and $\MHMso(X)_{[m,m+1]}$.  By Proposition~\ref{prop:winnow-beta}, the restriction of $\beta$ to these subcategories is an equivalence, so we have that
\begin{align*}
\beta&: \Ext^i_{\MHMs(X)}(\cF,\cG) \to \Ext^i_{\MHMso(X)}(\beta\cF,\beta\cG)
&&\text{is an isomorphism for $i \le 1$.}
\end{align*}
By Lemma~\ref{lem:ghom-basic}\eqref{it:ghom-hexs}, the map $\hexs: \Ext^i_{\MHM(X)}(\cF,\cG) \to \gHom^i(\cF,\cG)$ is an isomorphism as well.  Combining these observations with the commutativity of Figure~\ref{fig:delta} and with~\eqref{eqn:ext1-real} for $\rso$ and $\rsc$, we see that~\eqref{eqn:rso-hbeta-isom-comp} is an isomorphism for $i \le 1$.

We now proceed by induction on $i$.  Consider the commutative diagram
\[
\vcenter{\xymatrix@C=15pt{
\displaystyle\bigoplus\gHom^1(\cF, \cK) \otimes \gHom^{i-1}(\cK, \cG) \ar[r] \ar[d]_{\rso \circ \beta} &
  \gHom^i(\cF, \cG) \ar[d]^{\rso \circ \beta} \\
\displaystyle\bigoplus \Hom^1_{\Kb\Pure(X)}(\beta\cF,\beta\cK) \otimes \Hom^{i-1}_{\Kb\Pure(X)}(\beta\cK,\beta\cG) \ar[r]  &
  \Hom^i_{\Kb\Pure(X)}(\beta\cF,\beta\cG) }}
\]
where both direct sums range over all simple objects $\cK \in \MHMs(X)$ with $\wt(\cK) = \wt(\cF) - 1$.  Both horizontal maps are surjective, by Lemmas~\ref{lem:ghom-surj} and~\ref{lem:homkb-surj}.  The left-hand vertical map is an isomorphism by induction, so the right-hand vertical map is at least surjective.  Using Lemma~\ref{lem:ghom-basic}\eqref{it:ghom-hexs} and our assumptions on weights, we have
\begin{multline*}
\gHom^i(\cF,\cG) \cong \Ext^i_{\MHM(X)}(\cF,\cG)\\ \cong \Hom_{\Pure(X)[\wt(\cF)]}(\cF,\cG[i]) \cong \Hom_{\Kb\Pure(X)}(\beta\cF,\beta\cG[i]).
\end{multline*}
Since $\gHom^i(\cF,\cG)$ and $\Hom^i_{\Kb\Pure(X)}(\beta\cF,\beta\cG)$ have the same dimension, the right-hand vertical map must also be an isomorphism.
\end{proof}  

\begin{prop}\label{prop:qhered}
Consider the objects $\Delta_s^\win = \beta(\Delta_s)$ and $\nabla_s^\win = \beta(\nabla_s)$.
\begin{enumerate}
\item If $X_t \subset \overline{X_s}$, then for all $n \in \Z$, we have\label{it:qh-ext1}
\[
\Ext_{\MHMso(X)}^1(\Delta_s^\win,\MC_t(n)) = \Ext_{\MHMso(X)}^1(\MC_t(n),\nabla_s^\win) = 0.
\]
\item For all $s,t \in S$ and all $n \in \Z$, we have\label{it:qh-ext2}
\[
\Ext_{\MHMso(X)}^2(\Delta_s^\win, \nabla_t^\win(n)) = 0.
\]
\item $\MHMso(X)$ has enough projectives (resp.~injectives), and each projective (resp.~injective) has a filtration by objects of the form $\Delta_s^\win(n)$ (resp.~$\nabla_s^\win(n)$).\label{it:qh-proj}
\item $\MHMso(X)$ has finite cohomological dimension.\label{it:qh-cohom}
\end{enumerate}
\end{prop}
\begin{proof}
For part~\eqref{it:qh-ext1}, in view of~\eqref{eqn:ext1-real} for $\rso$ and Lemma~\ref{lem:rso-hbeta-isom}, it suffices to show that $\gHom^1(\Delta_s,\MC_t(n)) = \gHom^1(\MC_t(n),\nabla_s) = 0$.  These groups are subquotients of $\uHom^1(\Delta_s,\MC_t)$ and $\uHom^1(\MC_t,\nabla_s)$, respectively, and the vanishing of the latter is clear.  A similar argument using~\eqref{eqn:ext2-real} and the vanishing of $\uHom^2(\Delta_s,\nabla_t)$ establishes part~\eqref{it:qh-ext2}.

We have just shown that in the category $\MHMso(X)$, the objects $\Delta_s^\win$ and $\nabla_s^\win$ satisfy graded versions of axioms~(4) and~(6) of~\cite[Section~3.2]{bgs}.  The remaining axioms are obvious, so the arguments of that section apply to $\MHMso(X)$.  In particular, part~\eqref{it:qh-proj} is a restatement of~\cite[Theorem~3.2.1]{bgs}, and part~\eqref{it:qh-cohom} is a restatement of~\cite[Corollary~3.2.2]{bgs}.
\end{proof}

\begin{thm}\label{thm:koszul}
$\MHMso(X)$ is a Koszul category.
\end{thm}
\begin{proof}
According to~\cite[Proposition~5.8]{ar}, it suffices to show that the realization functor $\rso: \Db\MHMso(X) \to \Kb\Pure(X)$ is an equivalence.  To do this, we use the method of~\cite[Corollary~3.3.2]{bgs}.  By Proposition~\ref{prop:qhered}\eqref{it:qh-cohom}, $\Db\MHMso(X)$ is generated by the projectives and by the injectives in $\MHMso(X)$, so it suffices to show that if $P \in \MHMso(X)$ is projective and $I \in \MHMso(X)$ is injective, then
\begin{equation}\label{eqn:koszul-projinj}
\Hom^i_{\Kb\Pure(X)}(P,I) = 0 \qquad\text{for $i > 0$.}
\end{equation}
In view of Proposition~\ref{prop:qhered}\eqref{it:qh-proj}, we see that~\eqref{eqn:koszul-projinj} would follow if we knew that
\begin{equation}\label{eqn:koszul-stdcostd}
\Hom_{\Kb\Pure(X)}^i(\Delta_s^\win, \nabla_t^\win(n)) = 0\qquad\text{if $i > 0$.}
\end{equation}
By Lemma~\ref{lem:rso-hbeta-isom}, it suffices to show that $\gHom^i(\Delta_s,\nabla_t(n)) = 0$, and this follows from the well-known fact that $\uHom^i(\Delta_s,\nabla_t) = 0$ for $i > 0$.
\end{proof}

Since $\rso$ is an equivalence, Lemma~\ref{lem:rso-hbeta-isom} implies the following.

\begin{cor}\label{cor:hbeta-win}
For all $\cF,\cG \in \MHMs(X)$ and all $i \ge 0$, the natural map $\hbeta: \gHom^i(\cF,\cG) \to \Ext^i_{\MHMso(X)}(\beta(\cF),\beta(\cG))$ is an isomorphism.\qed
\end{cor}

\section{Koszul gradings on perverse sheaves}
\label{sect:degr}

Unlike mixed Hodge modules in $\MHMs(X)$, objects of the winnowed category $\MHMso(X)$ do not come equipped with a notion of ``underlying perverse sheaf.''  It is a nontrivial task (and the main goal of this section) to construct a functor relating $\MHMso(X)$ to $\Pervs(X)$.  More specifically, we will construct a \emph{degrading functor} $\degr: \MHMso(X) \to \Pervs(X)$ in the sense of~\cite{bgs}.  This means that $\degr$ is exact and enjoys the following properties:
\begin{enumerate}
\item There is a natural isomorphism $\varepsilon: \degr \circ (1) \simto \degr$.
\item We have $\degr(\MC_s) \cong \IC_s$.
\end{enumerate}
We will also prove the following statement relating $\Ext$-groups in the two categories.  In the language of~\cite{bgs}, this says that the pair $(\degr, \varepsilon)$ is a \emph{grading} on $\Pervs(X)$.
\begin{enumerate}\setcounter{enumi}{2}
\item There is a natural isomorphism induced by $\varepsilon$:
\[
\bigoplus_{n \in \Z} \Ext^i(\cF,\cG(n)) \simto \Ext^i(\degr\cF,\degr\cG).
\]
\end{enumerate}
As an application, we will prove a formality result for $\Pervs(X)$ that generalizes one of the main results of~\cite{schnuerer}.  See Corollary~\ref{cor:dg-filt-isom}.

This section relies heavily on the language of ``differential graded graded'' (dgg) rings and modules.  A \emph{dgg-ring} will mean a bigraded ring $\ucE = \bigoplus \ucE^{m,n}$ equipped with a differential of degree $(1,0)$, i.e., a collection of maps $\ucE^{m,n} \to \ucE^{m+1,n}$ satisfying the usual properties (see, e.g.,~\cite[\S 2.2]{schnuerer}).  The term ``internal grading'' will refer to the grading indexed here by ``$n$.''  Thus, the internal grading on a dgg-ring or dgg-module can be forgotten to yield an ordinary dg-ring or dg-module.  For a dgg-ring $\ucE$, we write $\Dp(\rdggmod\ucE)$ for the derived category of perfect right dgg-modules over $\ucE$.

The strategy for the construction of $\degr$ is to first show that $\Db\MHMso(X)$ is described by a certain dgg-ring $\ucE$, and then that $\Db\Pervs(X)$ is described by the dg-ring $\cE$ obtained by forgetting the internal grading on $\ucE$.  The functor $\degr$ will be obtained from the forgetful functor $\For: \Dp(\rdggmod\ucE) \to \Dp(\rdgmod\cE)$.

\subsection{Square root of the Tate twist}
\label{subsect:squareroot}

For the arguments below, the fact that the Tate twist on $\MHMs(X)$ has degree~$2$ rather than degree~$1$ presents something of a technical annoyance.  To remedy this, we use the general construction explained in~\cite[Section~4.1]{bgs}.  Define $\MHMs(X)^\half$ to be the category
\[
\MHMs(X)^\half = \MHMs(X) \oplus \MHMs(X).
\]
We equip it with a ``square root of the Tate twist,'' defined by
\[
(\cF_1,\cF_2)(\half) = (\cF_2(1),\cF_1).
\]
We identify $\MHMs(X)$ with the full subcategory of $\MHMs(X)^\half$ consisting of objects of the form $(\cF,0)$.  The same construction applies to $\MHMso(X)$ as well.  There are functors
\[
\beta: \MHMs(X)^\half \to \MHMso(X)^\half,
\qquad
\exs: \MHMs(X)^\half \to \Pervs(X).
\]
Moreover, all the morphisms in Figure~\ref{fig:delta} extend to objects of $\MHMs(X)^\half$, and Proposition~\ref{prop:qhered} and Theorem~\ref{thm:koszul} hold for $\MHMso(X)^\half$.  For the remainder of the section, we will work with $\MHMs(X)^\half$ and $\MHMso(X)^\half$ rather than with $\MHMs(X)$ and $\MHMso(X)$.

As an example, in the following proof, it is easy to check that the integer $n$ must be even if $\cF \in \MHMs(X)$, but it is even easier to avoid the question.

\begin{lem}\label{lem:pseudo-beta}
An object $\cF \in \MHMs(X)^\half$ is $\exs$-projective if and only if $\beta(\cF)$ is a projective object of $\MHMso(X)^\half$.
\end{lem}
\begin{proof}
If $\cF$ is $\exs$-projective, then $\uHom^1(\cF,{-})$ vanishes, and then by Corollary~\ref{cor:hbeta-win}, so does $\Ext^1(\beta(\cF),{-})$.  If $\cF$ is \emph{not} $\exs$-projective, then there is some $\MC_s$ such that $\uHom^1(\cF,\MC_s) \ne 0$.  In particular, there is some $n$ such that $\gr_n \uHom^1(\cF,\MC_s) \ne 0$, and so $\gHom^1(\cF, \MC_s(-\nhalf)) \ne 0$.  Corollary~\ref{cor:hbeta-win} then shows us that $\beta(\cF)$ is not projective.
\end{proof}

\subsection{Formality for $\MHMso(X)^\half$}
\label{subsect:dgg-formal}

Let $\tilde L = \bigoplus_{s \in S} \MC_s(-\half \dim  X_s)$.  This is pure of weight~$0$, and it contains every simple object of weight~$0$ in $\MHMs(X)^\half$ as a direct summand.  We will also work with the related objects
\[
L = \exs(\tilde L) \cong \bigoplus_{s \in S} \IC_s \in \Pervs(X)
\qquad\text{and}\qquad
L_\win = \beta(\tilde L)  \in \MHMso(X)^\half.
\]
The next step is to choose a $\exs$-projective resolution $\tilde P^\bullet$ of $\tilde L$.  We claim that $\tilde P^\bullet$ may be chosen so that the following additional assumption holds: 
\begin{equation}\label{eqn:linear-resoln}
\text{Every simple quotient of $\tilde P^i$ is pure of weight $i$.}
\end{equation}
Indeed, since $\MHMso(X)^\half$ is already known to be Koszul, we could first choose a linear projective resolution of $L_\win$, and then, using Proposition~\ref{prop:pseudoproj} and Lemma~\ref{lem:pseudo-beta}, lift it term-by-term to $\MHMs(X)$.  We may also assume that $\tilde P^\bullet$ has finitely many nonzero terms, by Proposition~\ref{prop:qhered}\eqref{it:qh-cohom}.  Let us put
\[
P^\bullet = \exs(\tilde P^\bullet)
\qquad\text{and}\qquad
P^\bullet_\win = \beta(\tilde P^\bullet).
\]
Thus, $P^\bullet$ is a projective resolution of $L$ in $\Pervs(X)$, and $P_\win^\bullet$ is a projective resolution of $L_\win$ in $\MHMso(X)^\half$.

Form the dgg-algebra $\ucA$ whose bigraded components are given by
\[
\ucA^{m,n} = \bigoplus_{i + j = m} \Hom(P_\win^{-i}, P_\win^j(\nhalf)),
\]
and whose differentials are induced by those in $P_\win^\bullet$.  Next, for a complex $\cF^\bullet$ of objects in $\MHMso(X)^\half$, let $\grHom^\bullet(P_\win^\bullet,\cF^\bullet)$ be the right dgg-$\ucA$-module with
\[
\grHom^\bullet(P_\win^\bullet,\cF^\bullet)^{m,n} = \bigoplus_{i + j = m} \Hom(P_\win^{-i}, \cF^j(\nhalf)).
\]
The following lemma is a graded version of a standard result on dg-modules over the dg endomorphism ring of a complex of projectives.  See, for example,~\cite[Proposition~7 and Remark~8]{schnuerer}.  We omit the proof.

\begin{lem}\label{lem:dgg-proj-equiv}
The functor $\grHom^\bullet(P_\win^\bullet,{-}): \Db\MHMso(X)^\half \simto \Dp(\rdggmod\ucA)$ is an equivalence of categories.\qed
\end{lem}

Next, let $\ucE$ be the bigraded ring $H^\bullet(\ucA)$.  We have
\[
\ucE^{m,n} = \Ext^m(L_\win,L_\win(\nhalf)).
\]
By Theorem~\ref{thm:koszul}, $\ucE$ is ``pure,'' meaning that $\ucE^{m,n}$ vanishes unless $m = n$.  It is well known that a dgg-algebra whose cohomology is pure is \emph{formal}, i.e., quasi-isomorphic as a dg-ring to its cohomology.  In our case, we can be a bit more specific.  Since $\ucA^{m,n} = 0$ for $n > m$, the cohomology group $\ucE^{m,m}$ is naturally a subspace of $\ucA^{m,m}$.  In other words, $\ucE$ can be identified with a sub-dgg-ring of $\ucA$.  The following lemma, whose proof we omit, is similar to~\cite[Proposition~4]{schnuerer}; the idea goes back to Deligne~\cite{del:weil}.

\begin{lem}\label{lem:dgg-formal}
The functor $\Dp(\rdggmod\ucA) \simto \Dp(\rdggmod\ucE)$ induced by the inclusion $\ucE \to \ucA$ is an equivalence of categories. \qed
\end{lem}

\subsection{Formality for $\Pervs(X)$}
\label{subsect:filt-formal}

We would now like to prove similar statements for $\Pervs(X)$, but because the relevant dg-ring is only filtered and not bigraded, the formality statement requires considerably more work.

Recall that $P^\bullet$ is a projective resolution of $L \in \Pervs(X)$.  Form the dg-ring $\cA = \End^\bullet(P^\bullet)$.  Explicitly, this is a graded ring whose graded components are given by
\[
\cA^m = \bigoplus_{i + j = m} \Hom(P^{-i}, P^j) \cong \bigoplus_{i+j=n} \exs \uHom(\tilde P^{-i},\tilde P^j),
\]
with differential induced by that in the complex $P^\bullet$ as usual.  Next, for a complex $\cF^\bullet$ of objects in $\Pervs(X)$, consider the right dg-$\cA$-module given by
\[
\Hom^\bullet(P^\bullet,\cF^\bullet)^m = \bigoplus_{i+j = m} \Hom(P^{-i}, \cF^j),
\]
and with differential again defined as usual.  As in Lemma~\ref{lem:dgg-proj-equiv}, we have:

\begin{lem}\label{lem:dg-proj-equiv}
The functor $\Hom^\bullet(P^\bullet,{-}): \Db\Pervs(X) \simto \Dp(\rdgmod\cA)$ is an equivalence of categories.\qed
\end{lem}

Next, let $\cE$ be the graded ring $H^\bullet(\cA)$.  We have
\[
\cE^m = \Ext^m(L,L).
\]
 The complex $\cA$ is naturally equipped with a filtration (the weight filtration on the terms $\uHom(\tilde P^{-i}, \tilde P^j)$), but not with a bigrading.  Thus, in contrast with the situation in the preceding section, $\cE$ cannot readily be identified with a subring of $\cA$, and a more delicate construction is needed to obtain an equivalence of categories.

\begin{lem}\label{lem:filt-formal}
There is a $(\cE,\cA)$-bimodule $M$ such that the functor
\begin{equation}\label{eqn:keller-tensor}
{-} \otimes^L_{\cE} M: \Dp(\rdgmod\cE) \to \Dp(\rdgmod\cA)
\end{equation}
is an equivalence of categories.
\end{lem}
\begin{proof}
Let $M$ be the right dg-$\cA$-module $\Hom^\bullet(P^\bullet,L)$.  Its cohomology $H^\bullet(M)$ is clearly isomorphic to $\Ext^\bullet(L,L)$.

We first claim that $M$ has zero differential, and so can be identified with its cohomology.  Indeed, $M$ can be obtained by applying $\exs$ to the complex of mixed Hodge structures $\uHom^\bullet(\tilde P^\bullet, \tilde L)$.  Clearly, $\uHom(\tilde P^{-i}, \tilde L[j]) = 0$ if $j \ne 0$, so
\[
\uHom^\bullet(\tilde P^\bullet, \tilde L)^n = \uHom(\tilde P^{-n}, \tilde L),
\]
and this is pure of weight $n$.  Since the differentials in $\uHom^\bullet(\tilde P^\bullet, \tilde L)$ respect the weight filtration, they must all vanish.  We therefore have $M = H^\bullet(M)$, and hence an isomorphism $M \cong \Ext^\bullet(L,L)$.

Via that isomorphism, we can equip $M$ with a left action of the ring $\cE = \Ext^\bullet(L,L)$, so we have a homomorphism of dg-rings
\begin{equation}\label{eqn:keller-crit}
\cE \to \End^\bullet_{\rdgmod\cA}(M).
\end{equation}
The criterion given in~\cite[Lemma~6.1(a)]{keller} says that~\eqref{eqn:keller-tensor} is an equivalence if the following two conditions hold: (1)~the map~\eqref{eqn:keller-crit} is a quasi-isomorphism, and (2)~the direct summands of $M$ generate $\Dp(\rdgmod\cA)$ as a triangulated category.

As a left $\cE$-module, $M$ is clearly free of rank~$1$.  It follows that~\eqref{eqn:keller-crit} is injective.

The augmentation map $P^\bullet \to L$ induces a map of right dg-$\cA$-modules $f: \cA \to M$.  This is a quasi-isomorphism, and since $M = H^\bullet(M)$, it is also surjective.  Thus, $M$ is generated as a right dg-$\cA$-module by the element $f(1_\cA)$, where $1_\cA$ is the identity element of $\cA$.  It follows that any endomorphism $M \to M$ of right dg-$\cA$-modules is determined by the image of $f(1_\cA)$, and in particular, that
\[
\dim \End^i_{\rdgmod\cA}(M) \le \dim M^i = \dim \cE^i.
\]
Since~\eqref{eqn:keller-crit} is injective, we now see that it is in fact an isomorphism (and, a fortiori, a quasi-isomorphism).

Finally, the category $\Dp(\rdgmod\cA)$ is generated as a triangulated category by the direct summands of the free module $\cA$.  Since $f: \cA \to M$ is a quasi-isomorphism, $\Dp(\rdgmod\cA)$ is also generated by the direct summands of $M$, as desired.
\end{proof}

\subsection{Construction of the degrading functor}

The theorem below, which is the main result of this section, is stated for $\MHMso(X)^\half$.  However, it is easy to see that it implies the corresponding result for $\MHMs(X)$ (involving only integer Tate twists), in the form described at the beginning of Section~\ref{sect:degr}.

\begin{thm}\label{thm:halfdegr}
There is an exact functor $\degr: \MHMso(X)^\half \to \Pervs(X)$ with the following properties:
\begin{enumerate}
\item There is a natural isomorphism $\varepsilon: \degr \circ (\half) \simto \degr$.\label{it:degr-twist}
\item We have $\degr(\MC_s) \cong \IC_s$.\label{it:degr-simple}
\item There is a natural isomorphism induced by $\varepsilon$:\label{it:degr-degr}
\[
\bigoplus_{n \in \Z} \Ext^i(\cF,\cG(\nhalf)) \simto \Ext^i(\degr\cF,\degr\cG).
\]
\end{enumerate}
The pair $(\degr,\varepsilon)$ is unique up to isomorphism, but not canonical.
\end{thm}
The uniqueness asserted at the end is actually a general property of Koszul gradings, as explained in~\cite[Lemma~4.3.3]{bgs} and the remark following it.  
\begin{proof}
We define $\degr$ at the level of derived categories to be the composition
\begin{multline*}
\Db\MHMso(X) \xrightarrow{\text{Lemma~\ref{lem:dgg-proj-equiv}}} \Dp(\rdggmod\ucA) \xrightarrow{\text{Lemma~\ref{lem:dgg-formal}}} \Dp(\rdggmod\ucE) \xrightarrow{\For} \\ \Dp(\rdgmod\cE) \xrightarrow{\text{Lemma~\ref{lem:filt-formal}}} \Dp(\rdgmod\cA) \xrightarrow{(\text{Lemma~\ref{lem:dg-proj-equiv}})^{-1}} \Db\Pervs(X).
\end{multline*}
Here, the last functor is the inverse to the equivalence in Lemma~\ref{lem:dg-proj-equiv}.  The first two functors commute with Tate twist, and $\For \circ (\half) \cong \For$, so part~\eqref{it:degr-twist} of the theorem follows.  It is easy to see by tracing through the definitions that $\degr(L_\win) \cong L$, and part~\eqref{it:degr-simple} can be deduced from this.  For part~\eqref{it:degr-degr}, we observe that an analogous property holds for $\For$, and that every other functor above is an equivalence.
\end{proof}

\begin{cor}\label{cor:dg-filt-isom}
The dg-ring $\cA$ of Section~\ref{subsect:filt-formal} is formal.  As a consequence, there is an equivalence of triangulated categories
\[
\dsp(X) \cong \Dp(\rdgmod\cE).
\]
\end{cor}
This equivalence of categories generalizes~\cite[Theorem~3]{schnuerer}; see Remark~\ref{rmk:schnuerer}.
\begin{proof}
Let $Q^\bullet = \degr(P^\bullet_\win)$.  By Theorem~\ref{thm:halfdegr}\eqref{it:degr-degr}, the dg-ring $\End^\bullet(Q^\bullet)$ is obtained by forgetting the internal grading on the dgg-ring $\ucA$ of Section~\ref{subsect:dgg-formal}.  In particular, we know from Lemma~\ref{lem:dgg-formal} that $\End^\bullet(Q^\bullet)$ is formal.  On the other hand, $Q^\bullet$ and $P^\bullet$ are both minimal projective resolutions of $L$, so they are (noncanonically) isomorphic.  It follows that $\cA \cong \End^\bullet(Q^\bullet)$, so $\cA$ is formal as well.

The formality of $\cA$ implies that there is an equivalence of triangulated categories $\Dp(\rdgmod\cA) \cong \Dp(\rdgmod\cE)$, similar to that in Lemma~\ref{lem:dgg-formal}.  Combining this with Lemma~\ref{lem:dg-proj-equiv} yields the last assertion above.
\end{proof}

\section{Compatibility of rational structure and grading}
\label{sect:compat}

We now have two exact, faithful functors from $\MHMs(X)$ to $\Pervs(X)$: $\exs$ and $\degr \circ \beta$.  It is natural to ask how they are related to one another.  In Section~\ref{subsect:modcat}, we set up some machinery for addressing this question.  Then, in Sections~\ref{subsect:real} and~\ref{subsect:bgs}, we apply this machinery to two settings where a comparison is possible.

\subsection{Module categories over endormorphism rings}
\label{subsect:modcat}

We retain the notation for the complexes $\tilde P^\bullet$, $P^\bullet$, and $P^\bullet_\win$ introduced in Section~\ref{sect:degr}, as well as for the complex $Q^\bullet = \degr(P^\bullet_\win)$ appearing in the proof of Corollary~\ref{cor:dg-filt-isom}.  Let us put
\[
\uE = \uEnd(\tilde P^0),
\quad
E = \End(P^0),
\quad
\bE = \bigoplus_{n \in Z} \gr_n \uEnd(\tilde P^0) \cong \bigoplus_{n \in \Z} \Hom(P^0_\win,P^0_\win(\nhalf))
\]
We regard $\uE$ as a ring object in the tensor category $\MHM_\cT(\pt)$ of mixed Hodge modules on a point.  (Here, and below, $\cT$ denotes the trivial stratification of a one-point space.)  Let $\MHmod\uE$ be the category right module objects over $\uE$ in $\MHM_\cT(\pt)$.  The ring $E$ can be identified with $\exs(\uE)$; as such, it inherits a filtration, and the associated graded ring is $\bE$.  Note that by taking the associated graded of the weight filtration of an object in $\MHmod\uE$, one obtains a graded $\bE$-module.  We denote this functor by $\gr: \MHmod\uE \to \rgmod \bE$.  Below, we will consider both graded and ungraded modules over $\bE$.  Let $\For: \rgmod \bE \to \rmod \bE$ be the forgetful functor which forgets the grading.

We will also require the following functors for passing from mixed Hodge modules or perverse sheaves to various module categories:
\begin{align*}
\ualpha_{\uE} &= \uHom(\tilde P^0,{-}): \MHMs(X)^\half \simto \MHmod\uE \\
\ualpha_{\bE} &= \bigoplus \Hom(P^0_\win(-\nhalf),-): \MHMso(X)^\half \simto \rgmod \bE \\
\alpha_E &= \Hom(P^0,-): \Pervs(X) \simto \rmod E \\
\alpha_{\bE} &= \Hom(Q^0,-): \Pervs(X) \simto \rmod \bE
\end{align*}
These are all equivalences of categories.  From the last two, we see that there exists an equivalence of categories
\[
\mu: \rmod E \simto \rmod \bE.
\]
One could take $\alpha_\bE \circ \alpha_E^{-1}$, but we will also make other choices of $\mu$ below.  In any case, given such a $\mu$, let $\tilde \mu = (\alpha_E)^{-1} \circ \mu \circ \alpha_{\bE}$.  We now assemble all these categories and functors into the diagram in Figure~\ref{fig:cube}.

\begin{figure}
\[
\vcenter{\xymatrix@=10pt{
\MHMs(X)^\half \ar[rrr]^{\beta} \ar[ddd]_{\ualpha_{\uE}} \ar[dr]_{\exs} &&&
  \MHMso(X)^\half \ar'[d][ddd]^{\ualpha_{\bE}} \ar[dr]^{\degr} \\
& \Pervs(X) \ar[rrr]^(.33){\tilde \mu} \ar[ddd]_(.33){\alpha_E} &&&
  \Pervs(X) \ar[ddd]^{\alpha_{\bE}} \\ \\
\MHmod\uE \ar'[r][rrr]_{\gr} \ar[dr]_{\exs} &&&
  \rgmod\bE \ar[dr]^{\For} \\
& \rmod E \ar[rrr]_{\mu} &&&
  \rmod \bE }}
\]
\caption{Comparison diagram for $\degr \circ \beta$ and $\exs$}\label{fig:cube}
\end{figure}

Let us now consider the various faces of this cube.  The front face is commutative by the definition of $\tilde \mu$.  We also know that the following three natural isomorphisms hold.  They say that the left, back, and right faces of Figure~\ref{fig:cube} commute.
\begin{align*}
\exs\uHom(\tilde P^0,{-}) &\cong \Hom(P^0, \exs({-})) \\
\bigoplus \gr_n\uHom(\tilde P^0,{-}) &\cong \bigoplus \Hom(P^0_\win(-\nhalf),\beta({-})) \\
\bigoplus \Hom(P^0_\win(-\nhalf),-) &\cong \Hom(Q^0,\degr({-})) 
\end{align*}
The commutativity of the top and bottom squares is more difficult to assess.  Since the vertical arrows are all equivalences, we can at least observe that the top square commutes if and only if the bottom square commutes.

A more precise version of the question posed at the beginning of Section~\ref{sect:compat} might then be: Does there exist a $\mu$ making the top and bottom faces of Figure~\ref{fig:cube} commute?

\subsection{Real coefficients}
\label{subsect:real}

In this section, we assume that $\F = \R$.  A key property of mixed Hodge structures over $\R$ is that the underlying vector space admits a functorial splitting of the weight filtration, sometimes called the \emph{Deligne splitting}.  More precisely: let $\gVect_\R$ be the category of graded finite-dimensional real vector spaces.  Let $\gr: \MHM_\cT(\pt) \to \gVect_\R$ be the functor which assigns to any mixed Hodge structure the associated graded vector space for the weight filtration, and let $\For: \gVect_\R \to \Vect_\R$ be the functor which forgets the grading.  We can identify $\Perv_\cT(\pt)$ with $\Vect_\R$ and regard $\exs$ as a functor $\MHM_\cT(\pt) \to \Vect_\R$.  Deligne's result~\cite{del:shmr} states there is an isomorphism of tensor functors
\begin{equation}\label{eqn:deligne}
\U: \exs \simto \For \circ \gr.
\end{equation}
Thus, when we apply $\U$ to the ring object $\uE$, we obtain a ring isomorphism
\begin{equation}\label{eqn:deligne-ring}
\U_{\uE}: E \simto \bE.
\end{equation}
Moreover, for any $M \in \MHmod\uE$, the isomorphism of vector spaces
\begin{equation}\label{eqn:deligne-mod}
\U_M: \exs(M) \simto \For(\gr(M))
\end{equation}
lets us regard the $E$-module $\exs(M)$ as an $\bE$-module, and this $\bE$-module structure coincides with that induced by~\eqref{eqn:deligne-ring}.

To rephrase this, let $\mu: \rmod E \to \rmod\bE$ be the equivalence induced by~\eqref{eqn:deligne-ring}.  Then, for every $M \in \MHmod\uE$, we can regard the map~\eqref{eqn:deligne-mod} as a natural isomorphism of $\bE$-modules, showing that $\mu \circ \exs \cong \For \circ \gr$.  In other words, the bottom square of Figure~\ref{fig:cube} commutes, and hence the top square does as well.

\begin{thm}\label{thm:real}
Suppose $\F = \R$.  Let $\mu$ be the equivalence induced by~\eqref{eqn:deligne-ring}.  Then there is a natural isomorphism $\mu \circ \exs \cong \degr \circ \beta$. \qed
\end{thm}

\begin{rmk}\label{rmk:schnuerer}
The isomorphism~\eqref{eqn:deligne} plays a vital role in Schn\"urer's proof of formality for constructible $\R$-sheaves on cell-stratified varieties~\cite[Theorem~31]{schnuerer}.  Specifically, it is the source of the vertical isomorphisms in the diagram in~\cite[Equation~(27)]{schnuerer}.  In Corollary~\ref{cor:dg-filt-isom} of the present paper, we have generalized this to arbitrary $\F \subset \R$ by a different argument, avoiding the use of~\eqref{eqn:deligne}.
\end{rmk}

\begin{rmk}\label{rmk:deligne-general}
The construction of~\eqref{eqn:deligne} in~\cite{del:shmr} relies on the fact that the only nontrivial element of $\mathrm{Gal}(\C/\R)$ is complex conjugation, and it seems unlikely that the construction can be generalized to other fields.  In more detail, Deligne first constructs an isomorphism $\U_\C$ of \emph{complex} vector spaces.  The map $\U_\C$ is unchanged when the Hodge filtration is replaced by its complex conjugate, and this implies that $\U_\C$ descends to the real vector space.   To mimic this for general $\F \subset \R$, one would need $\U_\C$ to be unchanged when the Hodge filtration is replaced by any conjugate under $\mathrm{Gal}(\C/\F)$, but that condition is false.  The authors do not know whether there is some other approach to~\eqref{eqn:deligne} that does work for general $\F$.
\end{rmk}

\subsection{Flag varieties and the Beilinson--Ginzburg--Soergel construction}
\label{subsect:bgs}

We now return to allowing arbitrary $\F$, and we turn our attention to the variety $X = G/B$, where $G$ is a complex reductive algebraic group and $B \subset G$ is a Borel subgroup.  Let $\cS$ be the stratification by $B$-orbits.  In this case, we will first make a  careful choice of $\tilde P^0$ and then replace $\MHMs(X)^\half$ by a certain full subcategory in such a way that the top face of Figure~\ref{fig:cube} commutes.

Let $X_e$ denote the unique $0$-dimensional $B$-orbit on $X$.  Then $\IC_e$ is a skyscraper sheaf, and $\MC_e$ is a simple mixed Hodge module of weight $0$.  Fix a $\exs$-injective envelope $\tilde I_e$ for $\MC_e$, and define a full subcategory of $\MHMs(X)$ as follows:
\[
\MHMs'(X;\tilde I_e) = \{ \cF \mid \text{$\cF$ is a subquotient of a direct sum of various $\tilde I_e(n)$} \}.
\]
Alternatively, $\MHMs'(X;\tilde I_e)$ can be described as the smallest full abelian subcategory of $\MHMs(X)$ containing $\tilde I_e$ and closed under subquotients and Tate twists.  This definition is dual to the one immediately preceding~\cite[Theorem~4.5.4]{bgs}, in the sense the latter is defined in terms of a $\exs$-projective instead.  It is straightforward to see that all the results of~\cite[Section~4.5]{bgs}, suitably dualized, hold for $\MHMs'(X;\tilde I_e)$, and we will invoke them in this way.

In particular, the main result of~\cite[Section~4.5]{bgs} states that $\MHMs'(X;\tilde I_e)$ is a Koszul category, and that the restriction of $\exs$ to this category makes it into a grading on $\Pervs(X)$.  In this last section, we seek to understand how this theorem is related to the main results of the present paper.  

\begin{rmk}\label{rmk:ginzburg}
As noted in the introduction, the results of~\cite[Section~4.5]{bgs} mentioned above are specific to $G/B$; they do not apply even to partial flag varieties.  However, the methods of~\cite{g} do apply to this case, as well as more generally for any smooth complex projective variety with a specific type of $\C^*$-action, giving an isomorphism of graded vector spaces which can be seen as a step towards the algebra isomorphism of parabolic--singular duality.
\end{rmk}

\begin{prop}\label{prop:bgs-equiv}
For any object $\tilde I_e$ as above, the functor
\[
\beta|_{\MHMs'(X;\tilde I_e)}: \MHMs'(X;\tilde I_e) \to \MHMso(X)
\]
is an equivalence of categories.
\end{prop}
\begin{proof}
For $\cF, \cG \in \MHMs'(X;\tilde I_e)$, the map
\begin{equation}\label{eqn:mhm-bgs}
\bigoplus_{n \in \Z} \Ext^i_{\MHMs'(X;\tilde\cP)}(\cF,\cG(n)) \to \Ext^i_{\Pervs(X)}(\exs\cF,\exs\cG)
\end{equation}
is an isomorphism by~\cite[Theorem~4.5.4]{bgs}.  Now, consider the map
\begin{equation}\label{eqn:mhm-bgs-beta}
\bigoplus_{n \in \Z} \Ext^i_{\MHMs'(X;\tilde\cP)}(\cF,\cG(n)) \to 
\bigoplus_{n \in \Z} \Ext^i_{\MHMso(X)}(\beta\cF, \beta\cG(n))
\end{equation}
induced by $\beta$.  Comparing~\eqref{eqn:mhm-bgs} to Theorem~\ref{thm:halfdegr}, we see that the domain and codomain of~\eqref{eqn:mhm-bgs-beta} have the same dimension.  That map is certainly injective for $i = 0$ (since $\beta$ is faithful), so it is an isomorphism.  In other words, we have that $\beta|_{\MHMs'(X;\tilde I_e)}$ is fully faithful.

We can therefore identify $\MHMs'(X;\tilde I_e)$ with a full subcategory of $\MHMso(X)$.  It follows immediately that~\eqref{eqn:mhm-bgs-beta} is also injective for $i = 1$, and hence (again for dimension reasons) an isomorphism.

Suppose that under this identification, $\MHMs'(X;\tilde I_e)$ does not coincide with $\MHMso(X)$, and let $\cF \in \MHMso(X)$ be an object of minimal length not belonging to $\MHMs'(X;I_e)$.  It follows from~\cite[Theorem~4.5.4]{bgs} that every simple object lies in $\MHMs'(X;I_e)$, so $\cF$ is not simple.  Thus, there is some short exact sequence $0 \to \cF' \to \cF \to \cF'' \to 0$ with $\cF'$ and $\cF''$ both nonzero.  Since they have shorter length than $\cF$, they both belong to $\MHMs'(X;I_e)$.  That short exact sequence represents a class in $\Ext^1_{\MHMso(X)}(\cF'', \cF')$ that is not in the image of the natural map
\[
\Ext^1_{\MHMs'(X;I_e)}(\cF'', \cF') \to \Ext^1_{\MHMso(X)}(\cF'', \cF').
\]
But this map was already seen to be an isomorphism, so we have a contradiction.
\end{proof}

An object $\cF \in \MHMs(X)$ is said to be \emph{special} if it is isomorphic to a subobject of a direct sum of objects of the form $\tilde I_e(n)$.  The following facts about special objects are proved in Steps~1 and~2 of the proof of~\cite[Theorem~4.5.4]{bgs}.

\begin{lem}\label{lem:bgs-special}
\begin{enumerate}
\item Every $\MC_s$ admits a special $\exs$-projective cover $\tilde P_s$, and this object is unique up to isomorphism.\label{it:bgs-proj}
\item If $\tilde P_s$ and $\tilde P_t$ are two special $\exs$-projective objects, then $\uHom(\tilde P_s,\tilde P_t)$ is a semisimple object of $\MHM_\cT(\pt)$.\label{it:bgs-semis}\qed
\end{enumerate}
\end{lem}

It follows from Proposition~\ref{prop:bgs-equiv} that special $\exs$-projective objects are in fact projective as objects of the abelian category $\MHMs'(X;\tilde I_e)$, and that all objects (not just simple ones) are quotients of special $\exs$-projective objects.  Thus, we may return to the construction of Section~\ref{sect:degr} and require that the $\exs$-projective resolution $\tilde P^\bullet \to \tilde L$ actually be a \emph{special} $\exs$-projective resolution.  The functor $\degr$ and the vertical arrows in Figure~\ref{fig:cube} should henceforth be understood to be defined with respect to this kind of resolution.

Let $\MHmodss \uE$ be the full subcategory $\MHmodss \uE$ consisting of $\uE$-modules that are semisimple as objects of $\MHM_\cT(\pt)$.  Since $\ualpha_{\uE}$ is exact, it follows from Lemma~\ref{lem:bgs-special}\eqref{it:bgs-semis} that $\ualpha_{\uE}(\cF) \in \MHmodss \uE$ for all $\cF \in \MHMs'(X;\tilde I_e)$.

It therefore makes sense to consider the diagram in Figure~\ref{fig:bgs}, which resembles Figure~\ref{fig:cube} but in which the domain and target of $\ualpha_{\uE}$ have been replaced by smaller categories.  It is easy to see that
\begin{equation}\label{eqn:bgs-ss-equiv}
\gr: \MHmodss \uE \simto \rgmod \bE
\end{equation}
is an equivalence of categories.  On the other hand,~\eqref{eqn:bgs-ss-equiv} gives rise to a canonical isomorphism of rings $E \simto \bE$, and hence to a canonical equivalence $\mu$ making the bottom face of Figure~\ref{fig:bgs} commute.  The top face then commutes as well.

\begin{figure}
\[
\vcenter{\xymatrix@=10pt{
\MHMs'(X;\tilde I_e) \ar[rrr]^{\beta} \ar[ddd]_{\ualpha_{\uE}} \ar[dr]_{\exs} &&&
  \MHMso(X)^\half \ar'[d][ddd]^{\ualpha_{\bE}} \ar[dr]^{\degr} \\
& \Pervs(X) \ar[rrr]^(.33){\tilde \mu} \ar[ddd]_(.33){\alpha_E} &&&
  \Pervs(X) \ar[ddd]^{\alpha_{\bE}} \\ \\
\MHmodss \uE \ar'[r][rrr]_{\gr} \ar[dr]_{\exs} &&&
  \rgmod\bE \ar[dr]^{\For} \\
& \rmod E \ar[rrr]_{\mu} &&&
  \rmod \bE }}
\]
\caption{Comparison diagram for the category $\MHMs'(X;\tilde I_e)$}\label{fig:bgs}
\end{figure}

To summarize, $\MHMs'(X; \tilde I_e)$ and $\MHMso(X)$ are equivalent Koszul categories via $\beta$, and the two Koszul gradings on $\Pervs(X)$ given by $\exs$ and $\degr$ are related by the following result.

\begin{thm}
There is a canonical autoequivalence $\tilde\mu$ of $\Pervs(X)$ giving rise to an isomorphism of functors $\tilde\mu \circ \exs \cong \degr \circ \beta: \MHMs'(X; \tilde I_e) \to \Pervs(X)$. \qed
\end{thm}



\end{document}